\documentclass[reqno]{amsart}

\usepackage{amssymb}  
\usepackage{amsxtra}    
\usepackage{graphicx,color}
\usepackage[mathscr]{eucal}

\allowdisplaybreaks

\newcommand{\set}[1]{\left\{#1\right\}}
\newcommand{\abs}[1]{\left\vert#1\right\vert}%
\newcommand{\norm}[1]{\left\Vert{#1}\right\Vert}%
\DeclareMathOperator{\supp}{\mathrm{supp}}
\theoremstyle{definition}
\newtheorem{defn}{Definition}[section]

\theoremstyle{theorem}
\newtheorem{thm}[defn]{Theorem}
\newtheorem{prop}[defn]{Proposition}
\newtheorem{cor}[defn]{Corollary}
\newtheorem{lem}[defn]{Lemma}

\theoremstyle{remark}
\newtheorem{rem}[defn]{Remark}
\numberwithin{equation}{section}

\usepackage[colorlinks,linkcolor=blue]{hyperref}

\begin{document}


\title[Multilinear Fourier Multipliers with Minimal Regularity, I]{Multilinear Fourier Multipliers with Minimal Sobolev Regularity, I}

\author[Grafakos]{Loukas Grafakos}
\address{Department of Mathematics, University of Missouri, Columbia, MO 65211, USA}
\email{grafakosl@missouri.edu}
\thanks{The first author would like to thank the Simons Foundation and the University of Missouri Research Council.}

\author[Nguyen]{Hanh Van Nguyen}
\address{Department of Mathematics, University of Missouri, Columbia, MO 65211, USA}
\email{hnc5b@mail.missouri.edu}
\thanks{The second author would like to thank Hue University - College of Education for their support.}

\begin{abstract}
We find optimal conditions
 on $m$-linear Fourier multipliers to give rise to bounded  operators from a product of Hardy spaces
 $H^{p_j}$, $0<p_j\le 1$,  to  Lebesgue spaces $L^p$.   The  conditions we obtain are necessary and sufficient
for boundedness and are expressed
 in terms of $L^2$-based Sobolev spaces. Our results
   extend those obtained in the linear case ($m=1 $) by Calder\'on and Torchinsky   \cite{CalTochin}
and in the bilinear case ($m=2$)
   by Miyachi and Tomita \cite{MiTo}.  We also prove a coordinate-type H\"ormander integral condition which we use to obtain certain extreme cases.
\end{abstract}

\maketitle

\section{Introduction}
Let $\sigma$ be a bounded function on $\mathbb R^{n}.$ We denote by $T_{\sigma}$ the linear Fourier multiplier operator, whose action on Schwartz functions is given by
\begin{equation}\label{equ:LiMulOp}
T_{\sigma}(f)(x) = \int_{\mathbb R^{n}} {\sigma}(\xi) \widehat{f}(\xi)e^{2\pi ix \xi}d\xi.
\end{equation}
Mikhlin's \cite{Miklin} classical result states that the $T_\sigma$ admits an $L^p$-bounded extension for $1<p<\infty,$
whenever
\begin{equation}\label{gghhtt}
{\left\vert{\partial^{\alpha}_{\xi}{\sigma}(\xi)}\right\vert} \le
C_{\alpha}{\left\vert{\xi}\right\vert}^{-{\left\vert{\alpha}\right\vert}},\quad \xi\ne 0
\end{equation}
for all multi-indices $\alpha$ with ${\left\vert{\alpha}\right\vert}\le [\frac{n}2]+1.$
This result was refined by H\"ormander \cite{Horman} who proved that \eqref{gghhtt} can be replaced by the
Sobolev-norm condition
\begin{equation}\label{gghhtt2}
\sup_{j\in\mathbb Z}\left\Vert{{\sigma}(2^{j}(\cdot))\widehat{\psi}}\right\Vert_{W^s}<\infty,
\end{equation}
for some $s>\frac{n}2,$ where $\widehat{\psi}$ is a smooth function supported in
$\frac12\le{\left\vert{\xi}\right\vert}\le 2$ that satisfies
$$
\sum_{j\in\mathbb Z}\widehat\psi(2^{-j}\xi) =1
$$
for all $\xi\ne 0.$ Here $\|g\|_{W^s} =\| (I-\Delta)^{s/2}g\|_{L^2}$, where $I$ is the identity operator and $\Delta= \sum_{j=1}^n \partial_j^2$, is the Laplacian on $\mathbb R^n$.

Calder\'on and Torchinsky   \cite{CalTochin} showed  that the   Fourier multiplier operator in \eqref{equ:LiMulOp}
admits a bounded extension from the Hardy space  $H^p$ to $H^p$ with  $0<p\le 1$ if
$$
\sup_{t>0}\left\Vert{{\sigma}(t\cdot)\widehat{\psi}}\right\Vert_{W^s}<\infty
$$
and $s>\frac{n}{p} - \frac{n}{2} $.
Here the index $s=\frac{n}{p} - \frac{n}{2}$ is critical in the sense that  the boundedness of $T_\sigma$ on $H^p$ does not hold if $s\le \frac{n}{p} - \frac{n}{2}$. This was pointed out later by Miyachi and Tomita  \cite{MiTo}.

The bilinear   counterpart of the Fourier multiplier theory has been rather similar in the formulation of  results, but
substantially  more complicated in   its proofs.  The theory of multilinear operators, and in particular that of multilinear multiplier operators, originated in the work of  Coifman and Meyer
\cite{CM1}, \cite{CM2}, \cite{CoiMey} and resurfaced in the work of Grafakos and Torres \cite{GraTor}.
Multilinear Fourier multipliers are bounded functions  $\sigma$ on $\mathbb R^{mn} = \mathbb R^n\times\cdots\times \mathbb R^n$  associated with the $m$-linear Fourier multiplier operator
\begin{equation}\label{equ:TSigmaMul}
T_{\sigma}(f_1,\ldots,f_m)(x) = \int_{\mathbb R^{mn}}e^{2\pi ix\cdot (\xi_1+\cdots+\xi_m)}\sigma(\xi_1,\ldots,\xi_m)\widehat{f_1}(\xi_1)\cdots \widehat{f_m}(\xi_m)\, d\vec \xi ,
\end{equation}
where  $f_j$ are in the Schwartz space of $\mathbb R^n$ and $d\vec \xi =d\xi_1\cdots d\xi_m  $.

Tomita \cite{TomitaJFA}
obtained $L^{p_1}\times \cdots \times L^{p_m}\to L^p$   boundedness ($1<p_1,\ldots,p_m,p< \infty$) for   multilinear multiplier operators
under a condition analogous to \eqref{gghhtt2}.  Grafakos and Si \cite{GraSi} extended Tomita's results to the case $p\le 1$ by using   $L^r$-based Sobolev norms for ${\sigma}$ with $1<r\le 2.$   Fujita and Tomita \cite{FuTomi} provided weighted  extensions of these results but also noticed that the Sobolev space $W^s$ in \eqref{gghhtt2} can be
replaced by a product-type Sobolev space $W^{(s_1,\dots ,s_m)}$ when $p> 2$.
Grafakos, Miyachi, and Tomita \cite{GrMiTo} extended the range of $p$ in \cite{FuTomi}
to $p>1$ and obtained boundedness even  in the endpoint case
where all but one indices $p_j$ are equal to infinity.
Miyachi and Tomita \cite{MiTo} provided    extensions of  the Calder\'on and Torchinsky  results \cite{CalTochin} for  Hardy spaces  in the bilinear case; it is noteworthy that in \cite{MiTo} it was pointed out  that the  conditions on the indices  are sharp, even in the linear case, i.e., in the  Calder\'on and Torchinsky theorem.

Following this stream of work, we are interested in finding conditions analogous to those in \cite{MiTo} in the multilinear setting, i.e., when $m\ge 3$. Our work is inspired by that of Calder\'on and Torchinsky \cite{CalTochin}, Grafakos and Kalton \cite{GrKal}, and certainly of Miyachi and Tomita \cite{MiTo}. As in \cite{MiTo}, we find necessary and sufficient conditions (which coincide with those in \cite{MiTo}) that imply boundedness for   multilinear multiplier operators on a products of   Hardy spaces. One important aspect of this work is an appropriate regularization of the multilinear multiplier
operator which allows the
interchange of its action with    infinite sums of $H^{p_j}$ atoms (see Section 3). In this article we restrict attention to the case where the domain is a product of Hardy spaces. We study the case where the domain is a mix of Lebesgue and
Hardy spaces in a subsequent article.

We   introduce  the Sobolev   spaces that will be used throughout   this paper.
First, for $x\in \mathbb R^n$ we set $\left<x\right> = \sqrt{1+|x| ^2}.$
For $s_1,\ldots,s_m>0,$ we denote by $W^{(s_1,\ldots,s_m)}$ the Sobolev space (of product type) consisting all   functions $f$ on $\mathbb R^{mn}$ such that
\begin{equation*}
\left\Vert{f}\right\Vert_{W^{(s_1,\ldots,s_m)}}:=
\left(
\int_{\mathbb R^{mn}}
\left\vert \widehat{f}(y_1,\ldots,y_m)
\left<y_1\right>^{s_1}\cdots \left<y_m\right>^{s_m}
\right\vert^2 dy_1\cdots dy_m
\right)^{\frac12} <\infty.
\end{equation*}
Notice that $W^{(s_1,\ldots,s_m)}$ is a subspace of $L^2$.

Let $\psi$ be a smooth function on $\mathbb R^{mn}$ whose Fourier transform $\widehat{\psi}$ is supported in $\frac12\le {\left\vert{\xi}\right\vert}\le 2$ and satisfies
\begin{equation*}
\sum_{j\in \mathbb Z}\widehat{\psi}(2^{-j}\xi) = 1,\qquad \xi\ne 0.
\end{equation*}
For $0<p<\infty$ we denote by $H^p$ the Lebesgue space $L^p$ if $p>1$ and the Hardy space $H^p$ if $p\le 1$.
The following is the main result of this paper.

\begin{thm}\label{thm:General}
  Let $\frac{n}2<s_1,\ldots,s_m<\infty,$ $0<p_1,\ldots,p_m\le 1,0< p< \infty$ such that
  $$\frac{1}{p_1}+ \cdots + \frac{1}{p_m} = \frac{1}{p},$$
  and that
  \begin{equation}\label{equ:SIndexCond}
  \sum_{k\in J}\Big(\dfrac{s_k}{n}-\dfrac1{p_k}\Big)>-\dfrac12
  \end{equation}
  for every subset $J\subset\left\{1,2,\ldots,m\right\}.$
  If the function ${\sigma}$ defined on $\mathbb R^{mn}$ satisfies
  \begin{equation}\label{equ:SigmCond}
  A:=\sup_{j\in\mathbb Z}\left\Vert{\sigma(2^j\cdot)\widehat\psi}\right\Vert_{W^{(s_1,\ldots,s_m)}}<\infty,
  \end{equation}
  then $T_{\sigma}$ is bounded from $H^{p_1}\times\cdots\times H^{p_m}\longrightarrow  L^p$ with constant at most a multiple of $A.$ Moreover, the set of {$2^m-1$} conditions \eqref{equ:SIndexCond}  is optimal.
\end{thm}


\begin{rem}
Conditions \eqref{equ:SIndexCond} imply that $s_i>\frac{n}2$ whenever $0<p_i\le 1$ for all $1\le i\le m.$ Moreover, the condition in \eqref{equ:SigmCond} is sufficient to guarantee that $\sigma\in L^{\infty}(\mathbb R^{mn}).$ Indeed, suppose that $\sigma$ is a function on $\mathbb R^{mn}$ that satisfies \eqref{equ:SigmCond}.
 It is easy to see that
  $\widehat\psi(\frac12x)+\widehat\psi(x)+\widehat\psi(2x)=1$
   for all $1\le x\le 2.$ Now we want to verify that ${\left\vert{\sigma(2^kx)}\right\vert}$ is uniformly bounded in $k$ for a.e. $1\le {\left\vert{x}\right\vert}\le 2.$ Applying the Cauchy-Schwarz  inequality and using the conditions $s_i>\frac{n}2,$ we have
  $$
  \begin{aligned}
  &{\left\vert{\sigma(2^kx)}\right\vert} \\
  =&  {\left\vert{\sum_{{\left\vert{l}\right\vert}\le 1}\sigma(2^{k}x)
  \widehat\psi(2^lx)}\right\vert}\\
   \le & \sum_{{\left\vert{l}\right\vert}\le 1}{\left\vert{\int_{\mathbb R^{mn}}\big(\sigma(2^{k-l}\cdot)
   \widehat\psi\big)\spcheck(\xi)e^{2^{l+1}\pi ix\xi}d\xi}\right\vert}\\
   \le & \sum_{{\left\vert{l}\right\vert}\le 1}\int_{\mathbb R^{mn}}
   \prod_{i=1}^m(1+{\left\vert{\xi_i}\right\vert}^2)^{-\frac{s_i}2} {\left\vert{\prod_{i=1}^m(1+{\left\vert{\xi_i}\right\vert}^2)^{\frac{s_i}2} \big(\sigma(2^{k-l}\cdot)\widehat\psi\big)\spcheck(\xi_1,\ldots,\xi_m)}\right\vert}d\xi_1\cdots d\xi_m\\
   \le & \sum_{{\left\vert{l}\right\vert}\le 1}C(s_1,\ldots,s_m,n)\left\Vert{\sigma_{k-l}}\right\Vert_{W^{(s_1,\ldots,s_m)}}\le 3C(s_1,\ldots,s_m,n)\sup_{j\in\mathbb Z}\left\Vert{\sigma_j}\right\Vert_{W^{(s_1,\ldots,s_m)}},
  \end{aligned}
  $$
  for almost all $x$ satisfying $1\le |x|\le 2.$ Thus
  $$
  \left\Vert{\sigma}\right\Vert_{L^{\infty}(\mathbb R^{mn})}\le 3C(s_1,\ldots,s_m,n)\sup_{j\in\mathbb Z}\left\Vert{\sigma_j}\right\Vert_{W^{(s_1,\ldots,s_m)}}<\infty .
  $$
\end{rem}

The structure of this paper is as follows: Section 2 contains preliminaries and   known results. In Section 3, we regularize the multiplier to be able to work with a nicer operator and thus facilitate the passage of infinite sums in and out the operator in the proof of the main result given in Section 4.
In section 5, we construct examples to justify the minimality of conditions \eqref{equ:SIndexCond} claimed  in the main theorem. Section 6 will present some results about the boundedness of our operator in the extreme cases where we need the Coordinate-type H\"ormander integral conditions. The last section contains the detail proof of some technical lemmas using through the paper.

\section{Preliminaries and known results}

Now fix $0<p<\infty$ and a Schwartz function $\Phi$ with $\widehat{\Phi}(0)\ne 0.$ Then the Hardy space $H^p$ contains all tempered distributions $f$ on $\mathbb R^n$ such that
$$
\left\Vert{f}\right\Vert_{H^p}:= \left\Vert{\sup_{0<t<\infty}{\left\vert{\Phi_t*f}\right\vert}}\right\Vert_{L^p}<\infty.
$$
It is well known that the definition of the Hardy space does not depend on the choice of the function $\Phi.$ Note that $H^p=L^p$ for all $p>1.$ When $0<p\le 1,$ one of nice features of Hardy spaces is the atomic decomposition. More precisely, any function $f\in H^p$ ($0<p\le 1$) can be decomposed as
$f=\sum_{k}\lambda_ka_k,$
where $a_k$'s are $L^\infty$-atoms for $H^p$ supported in
 cubes $Q_k$ such that $\left\Vert{a_k}\right\Vert_{L^{\infty}}\le {\left\vert{Q_k}\right\vert}^{-\frac1p}$
and $\int x^\gamma a_k(x)dx=0$
for all ${\left\vert{\gamma}\right\vert}\le \lfloor n(\frac1p-1)\rfloor+1,$ and the coefficients $\lambda_k$ satisfy
$\sum_k{\left\vert{\lambda_k}\right\vert}^p\le 2^p\left\Vert{f}\right\Vert_{H^p}^p.$

We also use notation $A\lesssim B$ to demonstrate that $A\le CB$ for some constant $C>0,$ and $A\approx B$ if $A\lesssim B$ and $B\lesssim A$ simultaneously.

The following two lemmas are essentially contained   in \cite{MiTo} modulo some minor modifications.
\begin{lem}[\cite{MiTo}]\label{lem:LInfL2}
      Let $k,l$ be positive integers, $s_i>0$   for $1\le i\le k+l$, and let $1<\rho<\infty.$ Assume that ${\sigma}$ is a function defined on $\mathbb R^{kn}\times \mathbb R^{ln},$ supported in $\big\{(x,y)\in\mathbb R^{kn}\times \mathbb R^{ln}\ :\ {\left\vert{x}\right\vert}^2+{\left\vert{y}\right\vert}^2\le 4\big\}$, where we denote $x=(x_1,\ldots,x_k), y=(y_1,\ldots,y_l)$ with $x_i,y_j\in\mathbb R^n,$ and set $K={\sigma}\spcheck,$ the inverse Fourier transform of $\sigma.$ Then there exists a constant $C>0$ such that
      $$
      \left\Vert{\left<y_1\right>^{s_{1}} \cdots\left<y_l\right>^{s_{l}}K(x,y)}\right\Vert_{L^{\infty}(\mathbb R^{ln},dy)} \le C \left\Vert{\left<y_1\right>^{s_{1}}\cdots \left<y_l\right>^{s_{l}}K(x,y)}\right\Vert_{L^{\rho}(\mathbb R^{ln},dy)}
      $$
for all $x\in \mathbb R^{kn}$.
\end{lem}
\begin{proof}
  Take $\varphi$ a Schwartz function on $\mathbb R^{ln}$ such that $\widehat{\varphi}(y) = 1$ for all $y\in \mathbb R^{ln},$ ${\left\vert{y}\right\vert}\le 2.$ Then we have ${\sigma}(x,y) = {\sigma}(x,y)\widehat{\varphi}(y).$ Using the inverse Fourier transform we have
  $$
  \begin{aligned}
  K(x,y) = \Big(K*(\delta_0\otimes {\varphi})\Big)(x,y)=& \int_{\mathbb R^{kn}\times \mathbb R^{ln}}K(x-u,y-v)\delta_0(u) {\varphi}(v)dudv\\
  =& \int_{\mathbb R^{ln}}K(x,y-v) {\varphi}(v)dv,
  \end{aligned}
  $$
  where $\delta_0$ is the Dirac distribution.  Therefore,
  $$
  \begin{aligned}
   & \left<y_1\right>^{s_{1}}  \cdots\left<y_l\right>^{s_{l}}{\left\vert{K(x,y)}\right\vert} \\
 =  &  \left<y_1\right>^{s_{1}} \cdots\left<y_l\right>^{s_{l}}
    {\left\vert{\int_{\mathbb R^{ln}} K(x_1,\ldots,x_k,y_1-v_1,\ldots,y_l-v_l) {\varphi}(v)dv }\right\vert}\\
    \lesssim& \int_{\mathbb R^{ln}}  \bigg( \prod_{j=1}^l \left<y_j-v_j\right>^{s_{j}}\bigg)
         {\left\vert{K(x_1,\ldots,x_k,y_1-v_1,\ldots,y_l-v_l)}\right\vert}
    \left<v_1\right>^{s_{1}} \cdots\left<v_l\right>^{s_{l}}{\left\vert{{\varphi}(v)}\right\vert}dv\\
    \le & C_1 \left\Vert{\left<y_1\right>^{s_{1}}\cdots \left<y_l\right>^{s_{l}}K(x,y)}\right\Vert_{L^{\rho}(\mathbb R^{ln},dy)}
    \left\Vert{\left<v_1\right>^{s_{1}} \cdots\left<v_l\right>^{s_{l}}{\left\vert{{\varphi}(v)}\right\vert}}\right\Vert_{L^{\rho'}(\mathbb R^{ln},dv)}\\
    \le & C_2\left\Vert{\left<y_1\right>^{s_{1}}\cdots \left<y_l\right>^{s_{l}}K(x,y)}\right\Vert_{L^{\rho}(\mathbb R^{ln},dy)},
  \end{aligned}
  $$
  where we used H\"older's inequality in the second to last line.
\end{proof}

\begin{lem}[\cite{MiTo}]\label{lem:ProdSobN}
  Let $s_i> \frac{n}{2}$ for $1\le i\le m,$ and let $\widehat{\zeta}$ be a smooth function which is supported in an annulus centered at zero. Suppose that $\Phi$ is a smooth function away from zero that satisfies the estimates
  $$
  {\left\vert{\partial^{\alpha}_{\xi}\Phi(\xi)}\right\vert}\le
  C_{\alpha}{\left\vert{\xi}\right\vert}^{-{\left\vert{\alpha}\right\vert}}
  $$
  for all $\xi\in \mathbb R^{mn},$ $x\ne 0$ and for all multi-indices $\alpha.$ Then there exists a constant $C$ such that
  $$
  \sup_{j\in\mathbb Z}\left\Vert{\sigma(2^{j}(\cdot))\Phi(2^j(\cdot))\widehat{\zeta}}\,\right\Vert_{W^{(s_1,\ldots,s_m)}}\le C\sup_{j\in\mathbb Z}\left\Vert{\sigma(2^{j}(\cdot))\widehat\psi}\, \right\Vert_{W^{(s_1,\ldots,s_m)}}.
  $$
\end{lem}
 Adapting the Calder\'on and Torchinsky  interpolation techniques in the multilinear setting (for  details on this we refer  to \cite[p. 318]{GrMiTo}) allows us to interpolate  between two endpoint estimates for multilinear multiplier operators from a product of some Hardy spaces to Lebesgue spaces.

\begin{thm}[\cite{GrMiTo}]\label{thm:CalTorInterp}
  Let $0<p_i,p_{i,k}\le \infty$ and $s_{i,k}>\frac{n}2$ for $i=1,2$ and $1\le k\le m.$ For $0<\theta<1$, set
  $\frac1{p}= \frac{1-\theta}{p_1}+\frac{\theta}{p_2},$ $\frac1{p_k} = \frac{1-\theta}{p_{1,k}}+\frac{\theta}{p_{2,k}},$ and $s_k = (1-\theta)s_{1,k}+\theta s_{2,k}.$
  Assume that the multilinear operator $T_{\sigma}$ defined in \eqref{equ:TSigmaMul} satisfies the estimates
  $$
  \left\Vert{T_{\sigma}}\right\Vert_{H^{p_{i,1}}\times\cdots\times H^{p_{i,m}}\longrightarrow  L^{p_i}}
  \le C_i\sup_{j\in\mathbb Z}\left\Vert{\sigma(2^j\cdot)\widehat\psi}\right\Vert_{W^{(s_{i,1},\ldots,s_{i,m})}},\qquad (i=1,2).
  $$
  Then
  $$
  \left\Vert{T_{\sigma}}\right\Vert_{H^{p_{1}}\times\cdots\times H^{p_{m}}\longrightarrow  L^{p}}
  \le C\sup_{j\in\mathbb Z}\left\Vert{\sigma(2^j\cdot)\widehat\psi}\right\Vert_{W^{(s_{1},\ldots,s_{m})}}.
  $$
\end{thm}
The following result is due to  Fujita
 and Tomita \cite{FuTomi} for
 $2 < p <\infty$, while the extension to $p>1$ and the endpoint case
where all but one indices  are equal to infinity is due to Grafakos, Miyachi and Tomita \cite{GrMiTo}.

\begin{thm}[\cite{FuTomi,GrMiTo}]\label{thm:PLarge}
  Let $1<p_1,\ldots,p_m\le \infty, 1<p<\infty$ and $\frac1{p_1}+\cdots+\frac1{p_m}=\frac1{p}.$ If $\sigma$ satisfies \eqref{equ:SigmCond}, then the multilinear multiplier operator $T_\sigma$ is bounded from $L^{p_1}\times\cdots\times L^{p_m}\longrightarrow  L^p$ with constant at most a multiple of $A.$
 \end{thm}

Finally, we will need the following lemma from \cite{GrKal}.

\begin{lem}[{\cite[Lemma 2.1]{GrKal}}]\label{lem:LpAvg}
  Let $0<p\le 1$ and let $\Big(f_Q\Big)_{Q\in \mathcal J}$ be a family of nonnegative integrable functions with $\supp(f_Q)\subset Q$ for all $Q\in\mathcal{J},$ where $\mathcal{J}$ is a family of finite or countable cubes in $\mathbb R^n.$ Then we have
  $$
  \left\Vert{\sum_{Q\in\mathcal{J}}f_Q}\right\Vert_{L^p}\lesssim \left\Vert{\sum_{Q\in\mathcal{J}}\Big(\frac1{{\left\vert{Q}\right\vert}}\int_Qf_Q(x)dx\Big) \chi_{Q^*}}\right\Vert_{L^p},
  $$
  with the implicit constant depending only on $p.$
\end{lem}

\section{Regularization the multiplier}
In this section, we   show that the operator defined in \eqref{equ:LiMulOp} with enough smoothness of the multiplier can be approximated by a family of very nice operators.

\begin{thm}\label{thm:RegMul}
  Let $\sigma$ be a function on $\mathbb R^{mn}$ and $s_i>\frac{n}2$ for $1\le i\le m$ satisfying \eqref{equ:SigmCond}.
  Then there exists a family of functions $\big(\sigma^{\epsilon}\big)_{0<\epsilon<\frac12}$ such that
    $K^{\epsilon}:=\big(\sigma^{\epsilon}\big)\spcheck$ is smooth and compactly supported for every $0<\epsilon<\frac12;$
    also
    \begin{equation}\label{equ:Eps2Ws}
      \sup_{0<\epsilon<\frac12}\sup_{j\in\mathbb Z}\left\Vert{\sigma^{\epsilon}(2^j\cdot)\widehat\psi}\right\Vert_{W^{(s_1,\ldots,s_m)}}
      \lesssim \sup_{j\in\mathbb Z}\left\Vert{\sigma(2^j\cdot)\widehat\psi}\right\Vert_{W^{(s_1,\ldots,s_m)}},
    \end{equation}
    and
    \begin{equation}\label{equ:T2Approx}
    \lim_{\epsilon\to 0}
    \left\Vert{T_\epsilon(f_1,\ldots,f_m)-T_{\sigma}(f_1,\ldots,f_m)}\right\Vert_{L^2} = 0
    \end{equation}
    for all functions $f_i\in L^{2m},$  $1\le i\le m$, where $T_{\epsilon}$ are multilinear singular integral operators of convolution type associated to $K^{\epsilon}.$
\end{thm}

The following lemma, whose proof will be given in the last section, is the first step in constructing such a family of functions $\sigma^{\epsilon}$ as stated in Theorem \ref{thm:RegMul}.
\begin{lem}\label{lem:RegSobN}
    Let $\varphi$ be a Schwartz function. Suppose $\sigma$ is a function on $\mathbb R^{mn}$ satisfying \eqref{equ:SigmCond} for $s_i>\frac{n}2.$ Then we have
    $$
    \sup_{\epsilon>0}\sup_{j\in\mathbb Z}\left\Vert{[(\varphi_\epsilon*\sigma)(2^j\cdot)]\widehat\psi}\right\Vert_{W^{(s_1,\ldots,s_m)}} \lesssim \sup_{j\in \mathbb Z}\norm{\sigma(2^j\cdot)\widehat{\psi}}_{W^{(s_1,\ldots,s_m)}}
    $$
    where $\varphi_{\epsilon}(x_1,\ldots,x_m) = \epsilon^{-mn}\varphi(\epsilon^{-1}x_1,\ldots,\epsilon^{-1}x_m)$ for all $x_i\in\mathbb R^{n}, 1\le i\le m.$
\end{lem}

We now start the proof of Theorem \ref{thm:RegMul}.
\begin{proof}[Proof of Theorem \ref{thm:RegMul}]
  Fix $0<\epsilon<\frac12.$ Choose a smooth function $\varphi$  such that $\widehat\varphi$ is supported in the unit ball and $\widehat\varphi(0)=1.$ Denote by $\sigma^\epsilon = \varphi_{\epsilon}*\big(\sigma\phi^{\epsilon}\big),$
  where $\phi^{\epsilon} = \theta(\epsilon^{-1}\cdot) - \theta(\epsilon\cdot)$, and $\theta$ is a smooth function satisfying $\theta(x) = 0$ for all $\abs{x}\le 1$ and $\theta(x) = 1$ for all $\abs{x}\ge 2.$
  We   note that these functions are suitable regularized versions of the multiplier in Theorem \ref{thm:RegMul}. Indeed, let $K^\epsilon = \big(\sigma^\epsilon\big)\spcheck = \big(\sigma\phi^{\epsilon})\spcheck\widehat\varphi(\epsilon(\cdot));$ then, $K^\epsilon$ are smooth functions with compact support for all $0<\epsilon<\frac12.$

 Using the fact that
 $$
| \partial^{\alpha}\phi^{\epsilon}(\xi)| \le C_{\alpha,\theta}\abs{\xi}^{-\alpha},\quad \xi\ne 0,\quad 0<\epsilon<\frac12.
 $$
 Lemma \ref{lem:RegSobN} applied to the function $\sigma\phi^{\epsilon}$ combined with Lemma \ref{lem:ProdSobN} gives
  \begin{align*}
    \sup_{0<\epsilon<\frac12}\sup_{j\in\mathbb Z}
    {\left\Vert\sigma^\epsilon(2^j\cdot)\widehat\psi\right\Vert}_{W^{(s_1,\ldots,s_m)}}\lesssim & \sup_{0<\epsilon<\frac12}\sup_{j\in\mathbb Z}
    \left\Vert{\sigma(2^j\cdot)\phi^{\epsilon}(2^j\cdot)\widehat\psi}
    \right\Vert_{W^{(s_1,\ldots,s_m)}}\\
    \lesssim&
    \sup_{j\in\mathbb Z}
    {\left\Vert\sigma(2^j\cdot)\widehat\psi\right\Vert}_{W^{(s_1,\ldots,s_m)}},
  \end{align*}
  which yields \eqref{equ:Eps2Ws}. Thus, we are left with establishing \eqref{equ:T2Approx}. For $\epsilon>0,$ now recall
      \begin{align*}
        T_\epsilon(f_1,\ldots,f_m)(x) = & \int K^{\epsilon}(x-y_1,\ldots,x-y_m)f_1(y_1)\cdots f_m(y_m)dy\\
        =& \int \sigma^{\epsilon}(\xi_1,\ldots,\xi_m)\widehat{f_1}(\xi_1)\cdots\widehat{f_m}(\xi_m)
        e^{2\pi ix(\xi_1+\cdots+\xi_m)}d\xi.
      \end{align*}
      Involve estimate \eqref{equ:Eps2Ws} with Theorem \ref{thm:PLarge}, we can see that $T_\sigma$ and $T_\epsilon$ are uniformly bounded from $L^{2m}\times \cdots\times L^{2m}\longrightarrow  L^2$ for all $0<\epsilon<\frac12.$ By density, it suffices to verify \eqref{equ:T2Approx} for all functions in the Schwartz class.

      Now fix Schwartz functions $f_i,$ for $ 1\le i\le m .$ The Fourier transform of \linebreak
      $T_\sigma(f_1,\ldots,f_m)$ can be written by
      \begin{align*}
        \int_{\mathbb R^{n(m-1)}}\!\!\!\!\!\!\!\sigma \big(\xi_1,\ldots,\xi_{m-1},\xi - \sum_{l=1}^{m-1}\xi_l\big)
        \widehat{f_1}(\xi_1)\!\cdots\!\widehat{f_{m-1}}(\xi_{m-1}) \widehat{f_m}\big(\xi-\sum_{l=1}^{m-1}\xi_l\big)
        d\xi_1\!\cdots \! d\xi_{m-1}.
      \end{align*}
      Similarly, the Fourier transform of $T_\epsilon(f_1,\ldots,f_m)$ is
      \begin{align*}
        \int_{\mathbb R^{n(m-1)}} \!\!\!\!\!\!\!\! \sigma^{\epsilon}\big(\xi_1,\ldots,\xi_{m-1},\xi - \sum_{l=1}^{m-1}\xi_l\big)
        \widehat{f_1}(\xi_1)\!\cdots\!\widehat{f_{m-1}}(\xi_{m-1}) \widehat{f_m}\big(\xi-\sum_{l=1}^{m-1}\xi_l\big)d\xi_1\!\cdots\! d\xi_{m-1}.
      \end{align*}
      We now claim that $\sigma^{\epsilon}$ converges pointwise to $\sigma.$ Take this claim for granted, we have
      $$
      \Big(T_{\epsilon}(f_1,\ldots,f_m)\Big)\sphat(\xi)\longrightarrow  \Big(T_{\sigma}(f_1,\ldots,f_m)\Big)\sphat(\xi),
      \quad \epsilon\to0
      $$
      for a.e. $\xi\in\mathbb R^n.$ Notice that
      $$
      \left\Vert{T_\epsilon(f_1,\ldots,f_m)-T_{\sigma}(f_1,\ldots,f_m)}\right\Vert_{L^2} =
      \left\Vert{\Big(T_\epsilon(f_1,\ldots,f_m)\Big)\sphat- \Big(T_{\sigma}(f_1,\ldots,f_m)\Big)\sphat{\,}}\, \right\Vert_{L^2}.
      $$
      Since $\left\Vert{\sigma^{\epsilon}}\right\Vert_{L^{\infty}}\lesssim \left\Vert{\sigma}\right\Vert_{L^{\infty}}<\infty$ for all $\epsilon>0,$ Lebesgue's dominated convergence theorem implies that
      $$
      \Big(T_\epsilon(f_1,\ldots,f_m)\Big)\sphat\longrightarrow  \Big(T_{\sigma}(f_1,\ldots,f_m)\Big)\sphat
      \quad \mbox{as}\ \epsilon\to 0
      $$
      in $L^2,$ and this establishes \eqref{equ:T2Approx}.

      It remains to prove the above claim about pointwise convergence of $\sigma^\epsilon$ as $\epsilon\to 0.$ For each $k\in \mathbb Z,$ we want to show that ${\sigma}^{\epsilon}(x)\longrightarrow  {\sigma}(x)$ for a.e. $2^{k}\le{\left\vert{x}\right\vert}\le 2^{k+1}.$ Indeed, let $0<\epsilon<\min\set{2^{2k-2},2^{-\abs{k}-2}}$ be an arbitrarily small positive number. Then we have
\begin{align*}
{\left\vert{{\sigma}^{\epsilon}(x)-{\sigma}(x)}\right\vert}
\le& \int_{{\left\vert{y}\right\vert}
\le \sqrt{\epsilon}}
{\left\vert{\varphi_{\epsilon}(y)}\right\vert}{\left\vert{{\sigma}(x-y)}\right\vert}
\sup_{2^{k}\le {\left\vert{x}\right\vert}\le 2^{k+1}}{\left\vert{\phi}^{\epsilon}(x-y)-1\right\vert}dy\\
&+\int_{{\left\vert{y}\right\vert}\le \sqrt{\epsilon}}{\left\vert{\varphi_{\epsilon}(y)}\right\vert}
{\left\vert{{\sigma}(x-y)-{\sigma}(x)}\right\vert}dy\\
&+\int_{{\left\vert{y}\right\vert}> \sqrt{\epsilon}}{\left\vert{\varphi_{\epsilon}(y)}\right\vert}
{\left\vert{{\sigma}(x-y){\phi}^{\epsilon}(x-y)-{\sigma}(x)}\right\vert}dy.
\end{align*}
The first integral vanishes since $\phi^{\epsilon}(x)=1$ for all $2\epsilon\le\abs{x}\le \frac1{\epsilon}.$
To estimate the second integral, we denote
$$
\widehat{\Psi}(x) = \sum_{\abs{j}\le 2}\widehat{\psi}( 2^{-j}x).$$
Then $\widehat{\Psi}(x) = 1$ for all $\frac14\le{\left\vert{x}\right\vert}\le 4.$ Therefore we have
$$
\widehat{\Psi}(2^{-k}(x-y))=\widehat{\Psi}(2^{-k}x) = 1
$$
for all $2^k\le{\left\vert{x}\right\vert}\le 2^{k+1}$ and ${\left\vert{y}\right\vert}\le 2^{k-1}.$
Now set $\sigma_j(x) = \sigma(2^jx)\widehat\psi(x)$ and estimate
\begin{align*}
 &\hspace{-.2in} \int_{{\left\vert{y}\right\vert}\le \sqrt{\epsilon}}{\left\vert{\varphi_{\epsilon}(y)}\right\vert}
  {\left\vert{{\sigma}(x-y)-{\sigma}(x)}\right\vert}dy \\
  =&\,\, \int_{{\left\vert{y}\right\vert}\le \sqrt{\epsilon}}{\left\vert{\varphi_{\epsilon}(y)}\right\vert}
  {\left\vert{{\sigma}(x-y)\widehat{\Psi}(2^{-k}(x-y))-{\sigma}(x)\widehat{\Psi}(2^{-k}x)}\right\vert}dy\\
  \le&\,\, \left\Vert{\varphi}\right\Vert_{L^1}\sup_{{\left\vert{y}\right\vert}\le \sqrt{\epsilon}}\left\Vert{{\sigma}(\cdot-y)\widehat{\Psi}(2^{-k}(\cdot-y))-{\sigma}\widehat{\Psi}(2^{-k}\cdot)}
  \right\Vert_{L^{\infty}}\\
  \le &\,\,\left\Vert{\varphi}\right\Vert_{L^1} \sum_{j=k-2}^{k+2}\sup_{{\left\vert{y}\right\vert}\le \sqrt{\epsilon}}\left\Vert{{\sigma}_j(\cdot-2^{-j}y)-{\sigma}_j}\right\Vert_{L^{\infty}}\\
  =&\,\,\left\Vert{\varphi}\right\Vert_{L^1} \sum_{j=k-2}^{k+2}\sup_{{\left\vert{y}\right\vert}\le 2^{-j}\sqrt{\epsilon}}\left\Vert{{\sigma}_j(\cdot-y)-{\sigma}_j}\right\Vert_{L^{\infty}}.
\end{align*}
We would like  to show
$$
\lim_{\epsilon\to 0}\sup_{{\left\vert{y}\right\vert}\le 2^{-j}\sqrt{\epsilon}}\left\Vert{{\sigma}_j(\cdot-y)-{\sigma}_j}\right\Vert_{L^{\infty}}=0.
$$
The preceding limit apparently converges to $0$ as $\epsilon\to 0$ because $\sigma_j\in W^{(s_1,\ldots,s_m)}$ for $s_i>\frac{n}{2},$  $1\le i\le m$. The last term is majorized by
$$
C\norm{\sigma}_{L^{\infty}}\int_{{\left\vert{y}\right\vert}\ge \frac1{\sqrt{\epsilon}}}{\left\vert{\varphi(y)}\right\vert}dy,
$$
which tends to $0$ when $\epsilon\to 0.$

Thus ${\sigma}^{\epsilon}(x)\longrightarrow  {\sigma}(x)$ as $\epsilon\to 0$ for a.e. $2^{k}\le{\left\vert{x}\right\vert}\le 2^{k+1}.$ Hence, ${\sigma}^{\epsilon}$ converges to ${\sigma}$ pointwise on $\mathbb R^{mn}.$ Also $\norm{{\sigma}^{\epsilon}}_{L^{\infty}(\mathbb R^{mn})}\lesssim \norm{{\sigma}}_{L^{\infty}(\mathbb R^{mn})}$ uniformly for all $\epsilon>0.$ The proof of Theorem \ref{thm:RegMul} is now complete.
\end{proof}

Now fix a Schwartz function $K.$ We denote the multilinear singular integral operator of convolution type associated with the kernel $K$   by
$$
{ T^K}(f_1,\ldots,f_m)(x) = \int_{\mathbb R^{mn}}K(x-y_1,\ldots,x-y_m)f_1(y_1)\cdots f_m(y_m)dy_1\cdots dy_m.
$$

\begin{prop}\label{prop:NiceMul}
  Let $K$ be a smooth function on $\mathbb R^{mn}$ with compact support.
  Then we have
  $$
  \norm{ {  T^K}  }_{H^{p_1}\times\cdots\times H^{p_m}\longrightarrow  L^p}\le C_K <{\infty}
  $$
  for all $0<p_1,\ldots,p_m,p<\infty$ and $\dfrac1{p}=\dfrac1{p_1}+\cdots+\dfrac1{p_m},$ where
  $T^K$ is the multilinear singular integral operator of convolution type associated with the kernel $K.$
\end{prop}
\begin{proof}
The boundedness of the operator $T^K$ can be deduced from \cite[Lemma 4.2]{GraDanq},
which provides the estimate (for some sufficiently large integer   $N$)
\begin{equation}\label{equ:TKEst}
{\left\vert{{ T^K}(f_1,\ldots,f_m)(x)}\right\vert}\lesssim \prod_{i=1}^m\mathcal{M}_N(f_i)(x),
\end{equation}
for all $f_i\in L^2\cap H^{p_i},$ in which
$$
  \mathcal{M}_{N}(f)(x) = \sup_{\varphi \in \mathfrak{F}_{N}}\sup_{t>0}\sup_{y\in B(x,t)}{\left\vert{(\varphi_t*f)(y)}\right\vert}
$$ is the grand maximal function with respect to $N,$ and
$$
\mathfrak{F}_{N} := \left\{\varphi \in \EuScript{S}(\mathbb R^n)\ :\
\int_{\mathbb R^n}(1+{\left\vert{x}\right\vert})^N\sum_{{\left\vert{\alpha}\right\vert}\le N+1}{\left\vert{\partial^{\alpha}\varphi(x)}\right\vert}dx\le 1\right\}.
$$
Taking $L^p$ norm, applying Holder's inequality to \eqref{equ:TKEst}, and using the quasi-norm equivalence of some maximal functions \cite[Theorem 6.4.4]{Grafk} give us
$$
\norm{T^K(f_1,\ldots,f_m)}_{L^p}\lesssim \prod_{i=1}^m\norm{\mathcal{M}_N(f_i)}_{L^{p_i}}\le C_K\prod_{i=1}^m\norm{f_i}_{H^{p_i}}.
$$
\end{proof}

Working with smooth kernels $K$ with compact support comes handy when working with infinite sums of atoms, since we are able to
freely interchange the action of $T^K$ with infinite sums of atoms. Precisely,
a consequence of the boundedness of $T^K$,  given in Proposition \ref{prop:NiceMul}, is the following result.

\begin{prop}\label{pro:SwitchSum}
  Let $0<p_1,\ldots,p_m\le 1$ and $0<p<\infty$ be numbers such that
   $$
   \frac1{p} = \frac1{p_1}+\cdots+\frac1{p_m},
   $$
   and let $K$ be a smooth function with compact support. Then for every $f_i\in H^{p_i}$ with atomic representation $f_i = \sum_{k_i}\lambda_{i,k_i}a_{i,k_i},$ where $a_{i,k_i}$ are $L^{\infty}$-atoms for $H^{p_i}$ and $\sum_{k_i}{\left\vert{\lambda_{i,k_i}}\right\vert}^{p_i}\le 2^{p_i}\norm{f_i}^{p_i}_{H^{p_i}}$ for $1\le i\le m.$ Then
  $$
 {  T^K}(f_1,\ldots,f_m)(x) = \sum_{k_1}\cdots\sum_{k_m}\lambda_{1,k_1}\cdots\lambda_{m,k_m} {  T^K}(a_{1,k_1},\ldots,a_{m,k_m})(x)
  $$
  for a.e. $x\in\mathbb R^n.$
\end{prop}

\begin{proof}
For any positive integers $N_1, \dots , N_m$ we have
$$
\begin{gathered}
\norm{{  T^K}(f_1,\ldots,f_m) - \sum_{k_1=1}^{N_1}\cdots\sum_{k_m=1}^{N_m} \lambda_{1,k_1}\cdots\lambda_{m,k_m}{ T^K}(a_{1,k_1},\ldots,a_{m,k_m})}_{L^p}\\\le C_K\sum_{i=1}^m\norm{f_i-\sum_{k_i=1}^{N_i}\lambda_{i,k_i}a_{i,k_i}}_{H^{p_i}}\prod_{l\ne i}\norm{f_l}_{H^{p_l}},
\end{gathered}
$$
so passing to the limit, we obtain
$$
{  T^K}(f_1,\ldots,f_m)(x) = \sum_{k_1=1}^{\infty}\cdots\sum_{k_m=1}^{\infty} \lambda_{1,k_1}\cdots\lambda_{m,k_m}{  T^K}(a_{1,k_1},\ldots,a_{m,k_m})(x)
$$
for a.e. $x\in\mathbb R^n.$
\end{proof}

\section{The proof of the main result}
In this section, we prove the main theorem. To do so, we first consider the case where $\sigma$ is smooth such that its Fourier's transform is compactly supported, then, by regularization, we can improve the result for any multiplier $\sigma$ in general.

We now start proving Theorem \ref{thm:General}.
\begin{proof}[Proof of the main theorem]
  By regularization, we may assume that the inverse Fourier transform of $\sigma$ is smooth and compactly supported. If this case is established, then Theorem \ref{thm:RegMul} yields the existence of a family of multilinear multiplier operators $\big(T_{\epsilon}\big)_{0<\epsilon<\frac12}$ associated with a family of multipliers $\big(\sigma^{\epsilon}\big)_{0<\epsilon<\frac12}$ such that $K^\epsilon=\big(\sigma^{\epsilon}\big)\spcheck$ are smooth functions with compact supports for all $0<\epsilon<\frac12,$ and that \eqref{equ:Eps2Ws}, \eqref{equ:T2Approx} hold. Fix $f_i\in H^{p_i}\cap L^{2m},$ $(1\le i\le m).$ The $L^2$ convergence in \eqref{equ:T2Approx} implies that we can find a sequence of positive numbers $\big(\epsilon_k\big)_{k}$ convergent to $0$ such that
  $$
  \lim_{k\to\infty}T_{\epsilon_k}(f_1,\ldots,f_m)(x) = T_{\sigma}(f_1,\ldots,f_m)(x)
  $$
  for a.e. $x\in\mathbb R^n.$ Fatou's lemma connecting with \eqref{equ:Eps2Ws} gives us
  \begin{align*}
    \norm{T_{\sigma}(f_1,\ldots,f_m)}_{L^p}\le & \liminf_{k\to\infty}\norm{T_{\epsilon_k}(f_1,\ldots,f_m)}_{L^p}\\
    \lesssim & \sup_{0<\epsilon<\frac12}\norm{T_{\epsilon}(f_1,\ldots,f_m)}_{L^p}\\
    \lesssim & \sup_{0<\epsilon<\frac12}\sup_{j\in\mathbb Z}\norm{\sigma^{\epsilon}(2^j\cdot)\widehat\psi\, }_{W^{(s_1,\ldots,s_m)}}
    \norm{f_1}_{H^{p_1}}\cdots\norm{f_m}_{H^{p_m}}\\
    \lesssim & \sup_{j\in\mathbb Z}{\left\Vert\sigma(2^j\cdot)\widehat\psi\, \right\Vert}_{W^{(s_1,\ldots,s_m)}}
    \norm{f_1}_{H^{p_1}}\cdots\norm{f_m}_{H^{p_m}},
  \end{align*}
  thus establishing the claimed estimate for a general multiplier $\sigma.$

In view of this deduction, we   suppose $\sigma\spcheck$ is smooth and compactly supported. The aim is to show that
  \begin{equation}\label{equ:TSigSmP}
    \norm{T_{\sigma}(f_1,\ldots,f_m)}_{L^p} \lesssim \sup_{j\in\mathbb Z}{\left\Vert\sigma(2^j\cdot)\widehat\psi\, \right\Vert}_{W^{(s_1,\ldots,s_m)}}
    \norm{f_1}_{H^{p_1}}\cdots\norm{f_m}_{H^{p_m}}.
  \end{equation}
  Fix functions $f_i\in H^{p_i}.$ Using atomic representations for $H^{p_i}$-functions, write
  $$
  f_i = \sum_{k_i\in\mathbb Z}\lambda_{i,k_i}a_{i,k_i},\quad (1\le i\le m),
  $$
  where $a_{i,k_i}$ are $L^\infty$-atoms for $H^{p_i}$ satisfying
  $$
  \supp(a_{i,k_i})\subset Q_{i,k_i},\quad \norm{a_{i,k_i}}_{L^{\infty}}\le {\left\vert{Q_{i,k_i}}\right\vert}^{-\frac{1}{p_i}},\quad \int_{Q_{i,k_i}}x^{\alpha}a_{i,k_i}(x)dx=0
  $$
  for all ${\left\vert{\alpha}\right\vert}$ large enough, and $\sum_{k_i}{\left\vert{\lambda_{i,k_i}}\right\vert}^{p_i}\le 2^{p_i}\norm{f_i}_{H^{p_i}}^{p_i}$.

  For the cube $Q,$ denote by $Q^*$ the dilation of the cube $Q$ with factor $2\sqrt{n}.$
  Since $K=\sigma\spcheck$ is smooth and compactly supported, Proposition \ref{pro:SwitchSum} yields that
  \begin{equation*}
  T_{\sigma}(f_1,\ldots,f_m)(x) = \sum_{k_1}\cdots\sum_{k_m}\lambda_{1,k_1}\ldots \lambda_{m,k_m} T_{\sigma}(a_{1,k_1},\ldots,a_{m,k_m})(x)
  \end{equation*}
  for a.e. $x\in\mathbb R^n.$  Now we can split $T_{\sigma}(f_1,\ldots,f_m)$ into two parts and estimate
  $$
  {\left\vert{T_{\sigma}(f_1,\ldots,f_m)(x)}\right\vert}\le G_1(x) +G_2(x),
  $$
  where $$G_1(x) = \sum_{k_1}\cdots\sum_{k_m}{\left\vert{\lambda_{1,k_1}}\right\vert}\ldots
  {\left\vert{\lambda_{m,k_m}}\right\vert} {\left\vert{T_{\sigma}(a_{1,k_1},\ldots,a_{m,k_m})}\right\vert}\chi_{Q^*_{1,k_1}\cap\ldots\cap Q^*_{m,k_m}}(x)$$
  and
  $$G_2(x) = \sum_{k_1}\cdots\sum_{k_m}{\left\vert{\lambda_{1,k_1}}\right\vert}\ldots
  {\left\vert{\lambda_{m,k_m}}\right\vert} {\left\vert{T_{\sigma}(a_{1,k_1},\ldots,a_{m,k_m})}\right\vert}\chi_{(Q^*_{1,k_1}\cap\ldots\cap Q^*_{m,k_m})^c}(x).$$
First we estimate the $L^p$-norm of $G_1,$ in which we   repeat the arguments in \cite{GrKal} for the sake of completeness. Without loss of generality, suppose   $Q^*_{1,k_1}\cap\ldots\cap Q^*_{m,k_m}\ne \emptyset$ and $Q_{1,k_1}$ has the smallest length among $Q_{1,k_1},\ldots,Q_{m,k_m}.$ Since $Q_{i,k_i}^*$ have the non-empty intersection, we can pick a cube $R_{k_1,\ldots,k_m}$ such that
  $$
  Q^*_{1,k_1}\cap\ldots\cap Q^*_{m,k_m} \subset R_{k_1,\ldots,k_m} \subset R_{k_1,\ldots,k_m}^*
  \subset Q^{\sharp}_{1,k_1}\cap\ldots\cap Q^{\sharp}_{m,k_m}
  $$
  and ${\left\vert{Q_{1,k_1}}\right\vert}\lesssim {\left\vert{R_{k_1,\ldots,k_m}}\right\vert},$
  where the implicit constant depends only on $n$ and $Q^{\sharp}_{i,k_i}$ denotes for a suitable dilation of $Q_{i,k_i}.$
  For $s_j>n/2,$ it was showed in \cite{GrMiTo} that
  $$
  \norm{T_{\sigma}}_{L^{2}\times L^{\infty}\times \cdots\times L^{\infty}\to L^2} \lesssim A.
  $$
  Therefore, by the Cauchy-Schwarz  inequality we have
  \begin{align*}
    \int_{R_{k_1,\ldots,k_m}}{\left\vert{T_{\sigma}(a_{1,k_1},\ldots,a_{m,k_m})(x)}\right\vert}dx \le& \norm{T_{\sigma}(a_{1,k_1},\ldots,a_{m,k_m})}_{L^2}{\left\vert{R_{k_1,\ldots,k_m}}\right\vert}^{\frac12}\\
    \lesssim & A{\left\vert{R_{k_1,\ldots,k_m}}\right\vert}^{\frac12}\norm{a_{1,k_1}}_{L^2}
    \prod_{i=2}^m\norm{a_{i,k_i}}_{L^\infty}\\
    \lesssim &
    A{\left\vert{R_{k_1,\ldots,k_m}}\right\vert}^{\frac12}{\left\vert{Q_{1,k_1}}\right\vert}^{\frac12}
    \prod_{i=1}^m{\left\vert{Q_{i,k_i}}\right\vert}^{-\frac1{p_i}}\\
    \lesssim &
    A{\left\vert{R_{k_1,\ldots,k_m}}\right\vert}\prod_{i=1}^m{\left\vert{Q_{i,k_i}}\right\vert}^{-\frac1{p_i}}.
  \end{align*}
  The last inequality implies that
  $$
  \frac1{{\left\vert{R_{k_1,\ldots,k_m}}\right\vert}}\int_{R_{k_1,\ldots,k_m}} {\left\vert{T_{\sigma}(a_{1,k_1},\ldots,a_{m,k_m})(x)}\right\vert}dx \lesssim A \prod_{i=1}^m{\left\vert{Q_{i,k_i}}\right\vert}^{-\frac1{p_i}}.
  $$
  Now the trivial estimate
  $$
  G_1(x) \le \sum_{k_1}\cdots\sum_{k_m}{\left\vert{\lambda_{1,k_1}}\right\vert}\ldots
  {\left\vert{\lambda_{m,k_m}}\right\vert} {\left\vert{T_{\sigma}(a_{1,k_1},\ldots,a_{m,k_m})}\right\vert}\chi_{R_{k_1,\ldots,k_m}}(x)
  $$
  combines with Lemma \ref{lem:LpAvg} to have
  \begin{eqnarray*}
    \norm{G_1}_{L^p} & \le &
    \norm{\sum_{k_1}\cdots\sum_{k_m}{\left\vert{\lambda_{1,k_1}}\right\vert}\ldots
    {\left\vert{\lambda_{m,k_m}}\right\vert} {\left\vert{T_{\sigma}(a_{1,k_1},\ldots,a_{m,k_m})}\right\vert}\chi_{R_{k_1,\ldots,k_m}}}_{L^p}\\
  &  \lesssim &
    A\norm{\sum_{k_1}\cdots\sum_{k_m}{\left\vert{\lambda_{1,k_1}}\right\vert}\ldots
    {\left\vert{\lambda_{m,k_m}}\right\vert} \Big(\prod_{i=1}^m{\left\vert{Q_{i,k_i}}\right\vert}^{-\frac1{p_i}}\Big)
    \chi_{R_{k_1,\ldots,k_m}^*}}_{L^p}\\
&    \le& A\norm{\sum_{k_1}\cdots\sum_{k_m}{\left\vert{\lambda_{1,k_1}}\right\vert}\ldots
    {\left\vert{\lambda_{m,k_m}}\right\vert} \prod_{i=1}^m\Big({\left\vert{Q_{i,k_i}}\right\vert}^{-\frac1{p_i}}\chi_{Q_{i,k_i}^{\sharp}}\Big)}_{L^p}\\
  &   = & A\norm{\prod_{i=1}^m\Big(\sum_{k_i}{\left\vert{\lambda_{i,k_i}}\right\vert} {\left\vert{Q_{i,k_i}}\right\vert}^{-\frac1{p_i}}\chi_{Q_{i,k_i}^{\sharp}}\Big)}_{L^p}\\
 &    \le & A\prod_{i=1}^m\norm{\sum_{k_i}{\left\vert{\lambda_{i,k_i}}\right\vert} {\left\vert{Q_{i,k_i}}\right\vert}^{-\frac1{p_i}}\chi_{Q_{i,k_i}^{\sharp}}}_{L^{p_i}} \\
 &    \lesssim & A\prod_{i=1}^m\norm{f_i}_{H^{p_i}}.
  \end{eqnarray*}
  Thus
  \begin{equation}\label{equ:G1Part}
  \norm{G_1}_{L^{p}}\lesssim A\norm{f_1}_{H^{p_1}}\cdots\norm{f_m}_{H^{p_m}}.
  \end{equation}
  Now for the harder part, $G_2(x),$ we first restrict $x\in (\cap_{i\notin J}Q_{i,k_{i}}^*)\setminus(\cup_{i\in J}Q_{i,k_{i}}^*)$ for some nonempty subset $J\subset\left\{1,2,\ldots,m\right\}.$ To continues our progress, we will need the following lemma whose proof will be given in {the last section.}
  \begin{lem}[The key lemma]\label{lem:KeyLem}
    Let $s_i>n/2$, $0<p_i,p\le 1$ be numbers and ${\sigma}$ be a function satisfying \eqref{equ:SIndexCond} and \eqref{equ:SigmCond}. Suppose $a_i$ are atoms supported in the cube $Q_i$, ($i=1,\ldots,m$) such that
   $$
   \norm{a_i}_{L^{\infty}}\le {\left\vert{Q_i}\right\vert}^{-\frac1{p_i}},\qquad \int_{Q_i}x^{\alpha}a_i(x)dx=0,
   $$
   for all ${\left\vert{\alpha}\right\vert}\le N_i$ with $N_i= \big[n(\frac1{p_i}-1)\big].$
   Fix a non-empty subset $J_0\subset\left\{1,\ldots,m\right\}.$
  Then there exist positive functions $b_1,\ldots,b_{m}$ such that
    \begin{equation}\label{equ:BiFuncs}
    {\left\vert{T_{\sigma}(a_1,\ldots,a_m)(x)}\right\vert} \lesssim A\,  b_1(x)\cdots b_{m}(x)
    \end{equation}
    for all $x\in (\cap_{i\notin J_0}Q_i^*)\setminus (\cup_{i\in J_0}Q_i^*),$
    and
    $\norm{b_i}_{L^{p_i}}\lesssim 1$,  $1\le i\le m .$
  \end{lem}
  Lemma \ref{lem:KeyLem} guarantees the existence of positive functions $b_{1,k_1}^J,\ldots,b_{m,k_m}^J$ depending on $Q_{1,k_1},\ldots,Q_{m,k_m}$ respectively, such that
  \begin{equation}\label{equ:BiJFns}
  {\left\vert{T_{\sigma}(a_{1,k_1},\ldots,a_{m,k_m})}\right\vert} \lesssim A\,  b_{1,k_1}^J \cdots b_{m,k_m}^J
  \end{equation}
  for all $x\in (\cap_{i\notin J}Q_{i,k_{i}}^*)\setminus(\cup_{i\in J}Q_{i,k_{i}}^*)$
  and $\norm{b_{i,k_i}^J}_{L^{p_i}} \lesssim 1.$ Now set
  $$
  b_{i,k_i} = \sum_{\emptyset \ne J\subset\left\{1,2,\ldots,m\right\}}b_{i,k_i}^J.
  $$
  Then
  \begin{equation}\label{equ:BiComp}
  {\left\vert{T_{\sigma}(a_{1,k_1},\ldots,a_{m,k_m})}\right\vert}\chi_{(Q^*_{1,k_1}\cap\ldots\cap Q^*_{m,k_m})^c}
  \lesssim A\,  b_{1,k_1} \cdots b_{m,k_m}
  \end{equation}
  and $\norm{b_{i,k_i}}_{L^{p_i}} \lesssim 1.$  Estimate   \eqref{equ:BiComp} yields
  $$
  G_2(x)\lesssim A\prod_{i=1}^m\left( \sum_{k_i}{\left\vert{\lambda_{i,k_i}}\right\vert}b_{i,k_i}(x) \right).
  $$
  Then apply   H\"older's inequality to deduce  that
  \begin{equation}\label{equ:G2Part}
  \norm{G_2}_{L^p}\lesssim A\, \norm{f_1}_{H^{p_1}}\cdots\norm{f_m}_{H^{p_m}}.
  \end{equation}
  Combining \eqref{equ:G1Part} and \eqref{equ:G2Part} yields   \eqref{equ:TSigSmP} as needed. The proof of Theorem \ref{thm:General} is now complete.
\end{proof}

\section{Minimality of   conditions}
In this section we will show that   conditions \eqref{equ:SIndexCond} and $s_i>\frac{n}2$ are
  minimal in general that guarantee   boundedness for multilinear multiplier operators.
  We fix a smooth function $\psi$ whose Fourier transform is supported in
$ \{2^{-\frac34}\le {\left\vert{\xi}\right\vert}\le 2^{\frac34} \},$ it satisfies $\widehat\psi(\xi) = 1$ for all $2^{-\frac14}\le {\left\vert{\xi}\right\vert}\le 2^{\frac14},$ and   for some nonzero constant $c$
$$
\sum_{j\in\mathbb Z}\widehat\psi(2^{-j}\xi) = c,\quad \xi\ne 0.
$$
Now we have the following theorem:
\begin{thm}\label{thm:NessConditions}
  Let $0<p_i\le\infty,$ $0<p<\infty,$ and $s_i>0$ for $1\le i\le m.$ Suppose that the estimate
  $$
\norm{T_\sigma(f_1,\ldots,f_m)}_{L^p}\lesssim \sup_{j\in\mathbb Z}{\left\Vert\sigma(2^j\cdot)\widehat\psi\right\Vert}_{W^{(s_1,\ldots,s_m)}}\prod_{i=1}^m\norm{f_i}_{H^{p_i}}
$$
holds for all $f_i\in H^{p_i}$ and $\sigma\in W^{(s_1,\ldots,s_m)}.$ The following conditions are then necessary:
\begin{equation}\label{equ:CondSi}
  s_i\ge \frac{n}2,\quad \forall 1\le i\le m,
\end{equation}
and
\begin{equation}
\label{equ:JMinCond}
\sum_{i\in J}\Big(\frac{s_i}{n}-\frac1{p_i}\Big)\ge -\frac{1}{2}
\end{equation}
for every subset $J\subset\left\{1,\ldots,m\right\}.$
\end{thm}
The following lemma is obvious by changing variables, so its proof is omitted.
\begin{lem}
  \label{lem:varphiepsilon}
  Let $\varphi$ be a nontrivial Schwartz function and $s>0.$ Then
  $$
  \int {\left\vert{\varphi(\epsilon y)}\right\vert}^2\big(1+{\left\vert{y}\right\vert}^2\big)^{s}dy \approx \epsilon^{-n-s}
  $$
  for all $0<\epsilon\le 1.$
\end{lem}
\begin{proof}[Proof of Theorem \ref{thm:NessConditions}]
We   show   first the necessary conditions \eqref{equ:CondSi} for $1\le i\le m.$ Without loss of generality, we will show $s_1\ge \frac{n}2.$
To establish this inequality, we need to construct some functions $\sigma^{\epsilon}, (0<\epsilon \ll 1),$ and $f_i\in H^{p_i}$ such that $\norm{f_i}_{H^{p_i}} = 1$ for all $1\le i\le m,$ and $\norm{T_{\sigma^{\epsilon}}(f_1,\ldots,f_m)}_{L^p} \approx 1,$ and further that
$$
\sup_{j\in\mathbb Z}\norm{\sigma^{\epsilon}(2^j\cdot)\widehat\psi}_{W^{(s_1,\ldots,s_m)}} \lesssim \epsilon^{\frac{n}2-s_1}.
$$
Once these functions are constructed, one may have
$$
1\approx \norm{T_{\sigma^{\epsilon}}(f_1,\ldots,f_m)}_{L^p} \lesssim \sup_{j\in\mathbb Z}\norm{\sigma^{\epsilon}(2^j\cdot)\widehat\psi}_{W^{(s_1,\ldots,s_m)}}\prod_{i=1}^m\norm{f_i}_{H^{p_i}} \lesssim \epsilon^{\frac{n}2-s_1}
$$ for all $0<\epsilon\ll 1.$ Therefore we get $s_1\ge \frac{n}2.$

Let $\varphi$ be a nontrivial Schwartz function such that $\widehat\varphi$ is supported in the unit ball, and let $\phi_2=\cdots=\phi_{m-1}$ be Schwartz functions whose Fourier transforms, $\widehat{\phi_2},$ is supported in an annulus $\frac1{17m}\le{\left\vert{\xi}\right\vert}\le \frac1{13m},$ and identical to $1$ on $\frac1{16m}\le{\left\vert{\xi}\right\vert}\le \frac1{14m}.$ Similarly, fix a Schwartz function $\phi_m$ with $\widehat{\phi_{m}}\subset \left\{\xi\in\mathbb R^n\ :\ \frac{12}{13}\le{\left\vert{\xi}\right\vert}\le\frac{14}{13}\right\}$ and $\widehat{\phi_m}\equiv 1$ on an annulus $\frac{25}{26}\le{\left\vert{\xi}\right\vert}\le\frac{27}{26}.$ Take $a,b\in\mathbb R^{n}$ with ${\left\vert{a}\right\vert}=\frac1{15m}$ and ${\left\vert{b}\right\vert}=1.$

For $0<\epsilon<\frac{1}{240m},$ set
$$
\sigma^{\epsilon}(\xi_1,\ldots,\xi_m) = \widehat\varphi\big(\frac{\xi_1-a}{\epsilon}\big)\widehat{\phi_2}(\xi_2)\cdots\widehat{\phi_m}(\xi_m).
$$
It is easy to check that $\supp{\sigma^\epsilon} \subset \left\{2^{-\frac14}\le{\left\vert{\xi}\right\vert}\le 2^{\frac14}\right\};$ hence, $\sigma^{\epsilon}(2^j\cdot)\widehat\psi = \sigma^{\epsilon}$ for $j=0$ and $\sigma^{\epsilon}(2^j\cdot)\widehat\psi = 0$ for $j\ne 0.$ This directly implies that
$$
\sup_{j\in\mathbb Z}\norm{\sigma^{\epsilon}(2^j\cdot)\widehat\psi}_{W^{(s_1,\ldots,s_m)}} = \norm{\sigma^{\epsilon}}_{W^{(s_1,\ldots,s_m)}}.
$$
Taking the inverse Fourier transform of $\sigma^{\epsilon}$ gives
$$
(\sigma^{\epsilon})\spcheck(x_1,\ldots,x_m) = \epsilon^n e^{2\pi ia\cdot x_1}\varphi(\epsilon x_1)\phi_2(x_2)\cdots\phi_m(x_m).
$$
Now apply Lemma \ref{lem:varphiepsilon} to have
$$
\norm{\sigma^{\epsilon}}_{W^{(s_1,\ldots,s_m)}} \lesssim \epsilon^{\frac{n}2-s_1}.
$$
Thus
$$
\sup_{j\in\mathbb Z}\norm{\sigma^{\epsilon}(2^j\cdot)\widehat\psi}_{W^{(s_1,\ldots,s_m)}} \lesssim \epsilon^{\frac{n}2-s_1}.
$$
Now choose $\widehat{f_i}(\xi) = \epsilon^{\frac{n}{p_i}-n}\widehat\varphi(\frac{\xi-a}{\epsilon})$ for $1\le i\le m-1,$ and $\widehat{f_m}(\xi) = \epsilon^{\frac{n}{p_m}-n}\widehat\varphi(\frac{\xi-b}{\epsilon}).$ Then we will show that these functions are what we needed to construct.

In the following estimates, we will use the fact, its proof can be done by using the Littlewood-Paley characterization for Hardy spaces, that if $f$ is a function whose Fourier transform is supported in a fixed annulus centered at the origin, then ${\left\Vert{f}\right\Vert}_{H^p}\approx {\left\Vert{f}\right\Vert}_{L^p}$ for $0<p<\infty.$

Indeed, using the above fact and checking that each $\widehat{f_i}$ is supported in an annulus centered at zero and not depending on $\epsilon$ allow us to estimate $H^p$-norms via $L^p$-norms. Namely
$$
\norm{f_i}_{H^{p_i}} \approx \norm{f_i}_{L^{p_i}}=1,\quad (1\le i\le m).
$$
Thus, we are left with showing that $\norm{T_\sigma(f_1,\ldots,f_m)}_{L^p} \approx 1.$ Notice that $\widehat{\phi_i}(\xi)=1$ on the support of $\widehat{f_i}$ for $2\le i\le m.$ Therefore,
\begin{align*}
T_{\sigma^{\epsilon}}(f_1,\ldots,f_m)(x)  \,=&\, \Big(\widehat{\varphi}\Big(\frac{\cdot -a}{\epsilon}\Big)\epsilon^{\frac{n}{p_1}-n}\widehat{\varphi}\Big(\frac{\cdot -a}{\epsilon}\Big)\Big)  \spcheck\! (x)
\Big(\widehat{\phi_2}\widehat{f_2}\Big)\spcheck\!(x)\cdots \Big(\widehat{\phi_m}\widehat{f_m}\Big)\spcheck\!(x)\\
 =&\,\Big(\widehat{\varphi}\Big(\frac{\cdot -a}{\epsilon}\Big)\epsilon^{\frac{n}{p_1}-n}\widehat{\varphi}\Big(\frac{\cdot -a}{\epsilon}\Big)\Big) \spcheck\!(x)
\Big(\widehat{f_2}\Big)\spcheck\!(x)\cdots \Big(\widehat{f_m}\Big)\spcheck\!(x)\\
 =&\,\,\epsilon^{\frac{n}{p_1}+\cdots+\frac{n}{p_{m}}}e^{2\pi i [(m-1)a+b]\cdot x}(\varphi*\varphi)(\epsilon x)
[\varphi(\epsilon x)]^{m-1}\\
 =&\,\, \epsilon^{\frac{n}{p}}e^{2\pi i [(m-1)a+b]\cdot x}(\varphi*\varphi)(\epsilon x)
[\varphi(\epsilon x)]^{m-1},
\end{align*}
which obviously gives $\norm{T_{\sigma^{\epsilon}}(f_1,\ldots,f_m)}_{L^p} \approx 1.$ So far, we have proved that $s_1\ge \frac{n}2;$ hence, by symmetry, we have $s_i\ge \frac{n}2$ for all $1\le i\le m.$

It now remains to show \eqref{equ:JMinCond}. By symmetry, we just only need to prove that
\begin{equation}\label{equ:S1Sr}
\sum_{i=1}^r\Big(\frac{s_i}{n}-\frac1{p_i}\Big)\ge -\frac{1}{2}
\end{equation}
for some fixed $1\le r\le m.$ To achieve our goal, we   construct a multiplier $\sigma^{\epsilon}$ such that
$$
\sup_{j\in\mathbb Z}\norm{\sigma^{\epsilon}(2^j\cdot)\widehat\psi}_{W^{(s_1,\ldots,s_m)}}
\lesssim \epsilon^{\frac{n}2-s_1-\cdots-s_r}
$$
for $0<\epsilon\ll 1$  and   functions $f_i$ satisfying $\norm{f_i}_{H^{p_i}} \approx 1$ for $1\le i\le m$ and
$$\norm{T_{\sigma^{\epsilon}}(f_1,\ldots,f_m)}_{L^p} \approx \epsilon^{n-\frac{n}{p_1}-\cdots-\frac{n}{p_r}}.$$
Then     inequalities
\begin{eqnarray*}
\epsilon^{n-\frac{n}{p_1}-\cdots-\frac{n}{p_r}}
&\approx &\norm{T_{\sigma^{\epsilon}}(f_1,\ldots,f_m)}_{L^p} \\
&\le & \sup_{j\in\mathbb Z}\norm{\sigma^{\epsilon}(2^j\cdot)\widehat\psi}_{W^{(s_1,\ldots,s_m)}}\prod_{i=1}^m\norm{f_i}_{H^{p_i}}\\
&\lesssim &\epsilon^{\frac{n}2-s_1-\cdots-s_r},
\end{eqnarray*}
for all small  positive numbers $\epsilon $ yield the claim \eqref{equ:S1Sr}.

We construct functions that give us enough ingredients to establish the multiplier $\sigma^{\epsilon}$ and functions $f_i,$ $(1\le i\le m),$ as mentioned above. Take two smooth functions $\varphi,\phi$ such that $\varphi(0)\ne 0,$ $\widehat\varphi$ is supported in $\left\{\xi\in\mathbb R^n\ :\ {\left\vert{\xi}\right\vert}\le \frac1{19mr}\right\}$ and $\widehat\varphi(x) =1$ for all ${\left\vert{\xi}\right\vert}\le \frac1{30mr},$ and that $\widehat\phi$ is supported in an annulus $\frac1{23m}\le {\left\vert{\xi}\right\vert}\le \frac1{19m}$ and $\widehat\phi(\xi) = 1$ for all $\frac1{22m}\le \left\vert{\xi}\right\vert\le \frac1{20m}.$ Fix $a,b\in\mathbb R^n$ such that $\left\vert{a}\right\vert=r^{-\frac12},$  $\left\vert{b}\right\vert=\frac1{21m}.$

For $0<\epsilon<\frac1{462m},$ define
\begin{align*}
\sigma^{\epsilon}&(\xi_1,\ldots,\xi_m) \\
 =&  \,\,\widehat\varphi\Big(\frac{1}{r\epsilon}\sum\limits_{i=1}^r(\xi_i-a)\Big)
\widehat\varphi\Big(\frac{1}{r}\sum\limits_{i=1}^r(\xi_i-\xi_2)\Big)\cdots
\widehat\varphi\Big(\frac{1}{r}\sum\limits_{i=1}^r(\xi_i-\xi_r)\Big)
\widehat\phi(\xi_{r+1})\cdots\widehat\phi(\xi_m).
\end{align*}
Once again, we have $\supp \sigma^{\epsilon}\subset  \{\xi\in\mathbb R^n\ :\ 2^{-\frac14}\le {\left\vert{\xi}\right\vert}\le 2^{\frac14}\},$ which, as in the previous case, implies that
$$
\sup_{j\in\mathbb Z}\norm{\sigma^{\epsilon}(2^j\cdot)\widehat\psi}_{W^{(s_1,\ldots,s_m)}} = \norm{\sigma^{\epsilon}}_{W^{(s_1,\ldots,s_m)}}.
$$
By changing variables, we can obtain the inverse Fourier transform of $\sigma^{\epsilon}$ as follows
\begin{align*}
(\sigma^{\epsilon})&\spcheck(x_1,\ldots,x_m) \\
 =&  \,\, re^{2\pi ia\cdot\sum\limits_{i=1}^rx_i}\epsilon^n\varphi\Big(\epsilon\sum_{i=1}^rx_i\Big)
\varphi(x_1-x_2)\cdots\varphi(x_1-x_r)\phi(x_{r+1})\cdots\phi(x_m).
\end{align*}
Taking Sobolev norm deduces
$$
\norm{\sigma^{\epsilon}}_{W^{(s_1,\ldots,s_m)}} = C\epsilon^n
\bigg(\int_{\mathbb R^{nr}}
{\left\vert{\varphi\Big(\epsilon\sum_{i=1}^rx_i\Big)
\prod_{i=2}^r\varphi(x_1-x_i)}\right\vert}^2
\prod_{i=1}^r(1+{\left\vert{x_i}\right\vert}^2)^{s_i}
dx_1\cdots dx_r
\bigg)^{\frac12},
$$
where $C= r\norm{\phi}_{W^{s_{r+1}}}\cdots\norm{\phi}_{W^{s_{m}}}.$

Next, we show that
\begin{align}\begin{split}\label{equ:VarpS1Sr}
\int_{\mathbb R^{nr}} &
{\left\vert{\varphi \Big(\epsilon\sum_{i=1}^rx_i\Big)
\varphi(x_1-x_2)\cdots\varphi(x_1-x_r)}\right\vert}^2
\prod_{i=1}^r(1+{\left\vert{x_i}\right\vert}^2)^{s_i}
dx_1\cdots dx_r \\
&\lesssim \epsilon^{-n-2(s_1+\cdots+s_r)}.
\end{split}\end{align}
In fact, changing variables in the above integral together with Lemma \ref{lem:varphiepsilon} yields
\begin{eqnarray*}
 && \int_{\mathbb R^{nr}}
{\left\vert{\varphi\Big(\epsilon\sum_{i=1}^rx_i\Big)
\varphi(x_1-x_2)\cdots\varphi(x_1-x_r)}\right\vert}^2
\prod_{i=1}^r(1+{\left\vert{x_i}\right\vert}^2)^{s_i}
dx_1\cdots dx_r\\
&= &\frac1{r}\int\limits_{\mathbb R^{nr}}
{\left\vert{\varphi(\epsilon y_1)
\varphi(y_2)\cdots\varphi(y_r)}\right\vert}^2
\left(1+\frac{1}{r^2}\Big|\sum_{i=1}^ry_i\Big|^2\right)^{s_1}  \\
&& \hspace{2in}\prod_{k=2}^r\left(1+\Big|-y_k+\frac1r\sum_{i=1}^ry_i\Big|^2\right)^{s_k}
dy_1\cdots dy_r\\
&\lesssim &\int\limits_{\mathbb R^{nr}}
{\left\vert{\varphi(\epsilon y_1)
\varphi(y_2)\cdots\varphi(y_r)}\right\vert}^2
\prod_{i=1}^r\Big(1+{\left\vert{y_i}\right\vert}^2\Big)^{s_1+\cdots+s_r}
dy_1\cdots dy_r\\
&\lesssim & \int\limits_{\mathbb R^{nr}}
{\left\vert{\varphi(\epsilon y_1)}\right\vert}^2
\Big(1+{\left\vert{y_1}\right\vert}^2\Big)^{s_1+\cdots+s_r}
dy_1 \\
&\lesssim &\epsilon^{-n-2s_1-\cdots-2s_r},
\end{eqnarray*}
where the implicit constants do not depend on $\epsilon.$
Inequality \eqref{equ:VarpS1Sr} gives us the estimate
$$
\sup_{j\in\mathbb Z}\norm{\sigma^{\epsilon}(2^j\cdot)\widehat\psi}_{W^{(s_1,\ldots,s_m)}} = \norm{\sigma^{\epsilon}}_{W^{(s_1,\ldots,s_m)}} \lesssim \epsilon^{\frac{n}2-s_1-\cdots-s_r}.
$$
To construct functions $f_i,$ we fix a smooth function $\zeta$ such that $\widehat\zeta$ is supported in $\left\{\xi\in\mathbb R^n\ :\ {\left\vert{\xi-a}\right\vert}\le \frac1{3m}\right\}$ and is identical to $1$ on   $\left\{\xi\in\mathbb R^n\ :\ {\left\vert{\xi-a}\right\vert}\le \frac3{19m}\right\}.$ Now set $f_1=\cdots=f_r=\zeta$ and $\widehat{f_i}(\xi) =\epsilon^{\frac{n}{p_i}-n}\widehat\varphi\big(\frac{\xi-b}{\epsilon}\big)$ for $r+1\le i\le m.$ It is clear     that
$$
\norm{f_i}_{H^{p_i}}\approx \norm{f_i}_{L^{p_i}} \approx 1,\quad 1\le i\le m.
$$
Moreover, $\widehat{f_1}(\xi_1)\cdots \widehat{f_r}(\xi_r)=1$ on the support of the function
$$
\widehat\varphi\Big(\frac{1}{r\epsilon}\sum\limits_{i=1}^r(\xi_i-a)\Big)
\widehat\varphi\Big(\frac{1}{r}\sum\limits_{i=1}^r(\xi_i-\xi_2)\Big)\cdots
\widehat\varphi\Big(\frac{1}{r}\sum\limits_{i=1}^r(\xi_i-\xi_r)\Big)
$$
and also $\widehat\phi(\xi)=1$ on the support of the functions $\widehat{f_i}$ for all $r+1\le i\le m.$
Therefore we have
$$
T_{\sigma^{\epsilon}}(f_1,\ldots,f_m)(x) = re^{2\pi i \big(ra+(m-r)b\big)\cdot x}\epsilon^n\varphi(\epsilon rx)\big[\varphi(0)\big]^{r-1}\epsilon^{\frac{n}{p_{r+1}}+\cdots+\frac{n}{p_m}}\big[\varphi(\epsilon x)\big]^{m-r}.
$$
Take $L^p$-norm, we get
$$
\norm{T_{\sigma^{\epsilon}}(f_1,\ldots,f_m)}_{L^p} \approx \epsilon^{n-\frac{n}{p_1}-\cdots-\frac{n}{p_r}},
$$
which is the last thing we want to obtain for our construction. Notice that the above argument also works for $p_i=\infty.$
\end{proof}
\section{Endpoint estimates}

In this section we consider  two endpoint estimates for  multilinear singular integral operators.
In the first case all indices are equal to infinity and in the second case one index is $1$ and the others are equal to infinity.

For $x\in\mathbb R^n$ and $1\le k\le m,$ define
$$
\Gamma^k_x = \left\{ (y_1,\ldots,y_m)\in\mathbb R^{mn}\ :\ |y_k|>2|x|\right\}.
$$
We say that a locally integrable function $K(y_1,\ldots,y_m)$   on $\mathbb R^{mn}\setminus\{0\}$ satisfies
 a {\it coordinate-type H\"ormander condition} if for some finite constant $A $ we have
\begin{equation}\label{equ:HormCond}
\sum_{k=1}^m\int_{\Gamma^k_x}\mid K(y_1,\ldots,y_{k-1},x-y_k,y_{k+1},\ldots,y_m) -K(y_1,\ldots,y_m) \mid\, d\vec{y}\le A
\end{equation}
for   all $x\in\mathbb R^{n}.$ Another type of (bi)-linear H\"ormander condition of geometric nature appeared in P\'erez and Torres \cite{PT}.

Denote by
$
\Lambda_p = \{(p,\infty,\ldots,\infty), (\infty,p,\infty,\ldots,\infty),\ldots,(\infty,\ldots,\infty,p)\}
$
the set of all $m$-tuples with $(m-1)$ entries equal to infinity and  only one entry equal to $p.$ The following result provides a version of the classical multilinear Calder\'on-Zygmund theorem in which the kernel
satisfies  a coordinate-type H\"ormander condition under the initial assumption that  the operator is  bounded on
Lebesgue spaces with indices in $\Lambda_2.$ We denote by $L^\infty_c$ the
space of all compactly supported bounded functions.

\begin{thm}\label{thm:EndPts}
Suppose that an $m$-linear singular integral operator of convolution type $T$ with kernel $K$ is bounded from $L^{q_1}\times \dots \times L^{q_m}$  to $L^2$ with norm at most $B$ for all $(q_1,\ldots,q_m)\in\Lambda_2.$
If $K$ satisfies the coordinate-type H\"ormander condition \eqref{equ:HormCond}, then
\begin{equation}\label{equ:LInf2BMO}
\| T(f_1, \dots , f_m) \|_{BMO} \lesssim  (A+B) \| f_1\|_{L^\infty} \cdots  \| f_m\|_{L^\infty}
\end{equation}
for all   $f_j$ in $L^\infty_c$. Moreover, $T$ has a bounded extension  which satisfies
\begin{equation}\label{equ:LWL1LInf}
\| T(f_1, \dots , f_m) \|_{L^{1,\infty}} \lesssim (A+B) \| f_i\|_{L^1}\prod_{\substack{i=1\\k\ne i}}^m\| f_k\|_{L^\infty}
\end{equation}
for all $1\le i\le m $,  $f_i\in L^1$, and $f_k\in L^\infty_c$ for $k\neq i$.
\end{thm}

\begin{proof}
Fix a cube $Q$. To prove \eqref{equ:LInf2BMO} we   show that there exists a constant $C_Q$ such that
\begin{equation}\label{equ:AvgQTfj}
 \frac1{|Q|}\int_{Q}\mid T(f_1,\ldots,f_m)(x)-C_Q\mid dx \lesssim (A+B) \| f_1\|_{L^\infty} \cdots  \| f_m\|_{L^\infty}.
\end{equation}

We decompose each function $f_j = f_j^0+f_j^1$, where $f_j^0=f_j \chi_{Q^*}$ and $f_j^1 = f_j\chi_{ (Q^* )^c}.$  Let $F$ be the set of the $2^m$ sequences of length $m$ consisting of zeros and ones.  We   claim that for each   sequence $\vec k=(k_1, \dots ,k_m)$ in
$F$ there is a constant $C_{\vec k}$ such that
\begin{equation}\label{equ:CkEST}
 \frac{1}{|Q|} \int_Q |T(f_1^{k_1}, \dots , f_m^{k_m})(x) -C_{\vec k}| \, dx \lesssim (A+B) \| f_1\|_{L^\infty} \cdots  \| f_m\|_{L^\infty}.
 \end{equation}
Assuming the validity of the preceding claim we obtain \eqref{equ:AvgQTfj} with $C_Q = \sum_{\vec k\in F} C_{\vec k}.$

Next, we want to establish \eqref{equ:CkEST} for each $\vec{k}\in F.$ If $\vec k = (k_1,\dots , k_m)$ has at least one zero entry   we pick $C_{\vec k}=0$. Without loss of generality, we may assume that $k_1=0.$ Since $T$ maps $L^{2}\times L^{\infty} \dots \times L^{\infty}$  to $L^2,$ we have
\begin{align*}
 \frac{1}{|Q|} \int_Q |T(f_1^{k_1}, \dots , f_m^{k_m})(x)  | \, dx & \le
 \bigg( \frac{1}{|Q|} \int_Q |T(f_1^{k_1}, \dots , f_m^{k_m})(x)  |^2 \, dx \bigg)^{\frac 12} \\
 & \le \bigg( \frac{1}{|Q|} \int_{\mathbb R^n} |T(f_1^{k_1}, \dots , f_m^{k_m})(x)  |^2 \, dx \bigg)^{\frac 12} \\
  & \le  B\, |Q|^{-\frac12} \| f_1^0\|_{L^{2}}\|f_2^{k_2}\|_{L^{\infty}} \cdots  \| f_m^{k_m}\|_{L^{\infty}}  \\
  & \le  B\, |Q|^{-\frac12} |Q^*|^{\frac 1{2}}  \| f_1\|_{L^\infty} \cdots  \| f_m\|_{L^\infty}  \\
 & \lesssim B \| f_1\|_{L^\infty} \cdots  \| f_m\|_{L^\infty} \, .
\end{align*}

Now suppose that
$\vec k = (1,\dots , 1)$. Set $C_{\vec k} =T(f_1^{k_1}, \dots , f_m^{k_m})(x_Q ),$ where $x_Q$ is the center of the cube $Q.$ Then we have
\begin{align*}
& \frac{1}{|Q|} \int_Q |T(f_1^{1}, \dots , f_m^{1})(x) -C_{\vec k} |dx \\
 & \le
  \frac{1}{|Q|} \int_Q \int_{\mathbb R^{nm}} |K(x\!-\!y_1,\ldots,x\!-\!y_m)-K(x_Q\!- \!y_1,\ldots,x_Q\!-\!y_m)| \prod_{i=1}^m|f_i^{1}(y_i)| d\vec ydx   \\
 & \le
\frac{\prod\limits_{i=1}^m\|f_i\|_{L^{\infty}}}{|Q|} \int_Q
 \sum_{k=1}^m\int_{\Gamma^k_{x-x_Q}} |K(y_1,\ldots,(x\!-\!x_Q)\!-\!y_k,\ldots,y_m)-K(y_1,\ldots,y_m)|d\vec ydx \\
 & \lesssim A \| f_1\|_{L^\infty} \cdots  \| f_m\|_{L^\infty} \, .
\end{align*}

This completes the proof of \eqref{equ:LInf2BMO} and we
are left with establishing \eqref{equ:LWL1LInf}. Fix $\lambda>0.$
It is enough to show that
$$
|\{x\in\mathbb R^{n}\ :\ |T(f_1,\ldots,f_m)(x)|>2\lambda\} | \lesssim (A+B)\frac{1}{\lambda}\norm{f_1}_{L^1}\norm{f_2}_{L^\infty}\cdots\norm{f_m}_{L^\infty}.
$$
By scaling, we may assume that $\norm{f_1}_{L^1}=\norm{f_2}_{L^{\infty}}=\cdots=\norm{f_m}_{L^{\infty}}=1.$ Let $\delta$ be a positive number chosen later and let $f_1=g_1+b_1$ be the Calder\'on-Zygmund decomposition at height $\delta\lambda,$ and $b_1 = \sum_{j}b_{1,j},$ where $b_{1,j}$ are functions supported in the (pairwise disjoint) cubes $Q_{j}$ such that
\begin{gather*}
\supp(b_{1,j})\subset Q_j,\quad \int b_{1,k}(x)dx = 0,\\
\norm{b_{1,j}}_{L^1} \le 2^{n+1}\delta\lambda |Q_{j}|,\quad \sum_{j}|Q_{j}|\le \frac1{\delta\lambda},\\
\norm{g_1}_{L^{\infty}}\le 2^{n}\delta\lambda,\quad \norm{g_1}_{L^1}\le 1.
\end{gather*}
Now we can estimate
\begin{align*}
 |\{x\in\mathbb R^{n}\ :\ |T(f_1,\ldots,f_m)(x)|>2\lambda\} |
 \le& |\{x\in\mathbb R^{n}\ :\ |T(g_1,\ldots,f_m)(x)|>\lambda\} |\\
 &+|\{x\in\mathbb R^{n}\ :\ |T(b_1,\ldots,f_m)(x)|>\lambda\} |.
\end{align*}
Since $T$ maps $L^2\times L^{\infty}\times\cdots\times L^{\infty},$ the first part can be controlled by
\begin{align*}
  |\{x\in\mathbb R^{n}\ :\ |T(g_1,\ldots,f_m)(x)|>\lambda\} |
  \le& \frac1{\lambda^2}\int_{\mathbb R^n}|T(g_1,f_2,\ldots,f_m)(x)|^2dx\\
  \le& \frac{B^2}{\lambda^2}\norm{g_1}_{L^2}^2\norm{f_2}_{L^{\infty}}^2\cdots\norm{f_m}_{L^{\infty}}^2\\
  \le& \frac{2^nB^2\delta}{\lambda^2}.
\end{align*}
To estimate the second part, we set $G = \cup_{j}Q^*_j.$ Then we have
\begin{align*}
  |\{x\in\mathbb R^{n}\ :\ |T(b_1,\ldots,f_m)(x)|>\lambda\} |
  \le& |G|+ |\{x\in G^c\ :\ |T(b_1,\ldots,f_m)(x)|>\lambda\} |\\
  \le& |G| + \frac1{\lambda}\sum_j\int_{\big(Q^*_j\big)^c}|T(b_{1,j},\ldots,f_m)(x)|dx.
\end{align*}
Notice that
$$
|G| \le \sum_j|Q^*_j| \lesssim \sum_{j}|Q_j|\le \frac1{\delta\lambda}.
$$
Denote by $c_j$ the center of the cube $Q_j.$ Invoking condition \eqref{equ:HormCond} yields
\begin{align*}
  &\int_{ (Q^*_j )^c}|T(b_{1,j},\ldots,f_m)(x)|dx\\
  \le& \int_{ (Q^*_j )^c}\bigg|\int K(x-y_1,\ldots,x-y_m)b_{1j}(y_1)f_2(y_2)\cdots f_m(y_m)d\vec{y} dx\bigg|\\
  \le& \int_{ (Q^*_j )^c}\!\!\bigg|\int \big[K(x\!-\!y_1,y_2,\ldots,y_m)-K(x\!-\!c_j,y_2,\ldots,y_m)\big]b_{1j}(y_1)\prod_{i=2}^mf_i(x\!-\!y_i)d\vec{y} dx\bigg|\\
  \le & \prod_{i=2}^m\norm{f_i}_{L^{\infty}}
  \int_{ (Q^*_j )^c} \int \big|K(x-y_1,y_2\ldots,y_m)-K(x-c_j,y_2,\ldots,y_m)\big|
  |b_{1j}(y_1)|d\vec{y} dx\\
  \le &
  \int_{Q_j}\left\{
  \int_{\Gamma^1_{y_1-c_j}} \big|K(y_1-z_1,z_2\ldots,z_m)-K(z_1,z_2,\ldots,z_m)\big|d\vec{z}\right\}|b_{1j}(y_1)|dy_1\\
  \le & A\norm{b_{1,j}}_{L^1}.
\end{align*}
Therefore
\begin{align*}
  \frac1{\lambda}\sum_j\int_{ (Q^*_j )^c}|T(b_{1,j},\ldots,f_m)(x)|dx
  \le & \frac{A}{\lambda}\sum_j\norm{b_{1,j}}_{L^1}
  \le \frac{2^{n+1}A}{\lambda}.
\end{align*}
Choosing $\delta = B^{-1}$ and combining the preceding inequalities  we obtain
$$
|\{x\in\mathbb R^{n}\ :\ |T(f_1,\ldots,f_m)(x)|>2\lambda\} |
\le \frac1{\lambda}(2^nB+B+2^{n+1}A)\le 2^{n+1}(A+B)\frac1{\lambda},
$$
which yields \eqref{equ:LWL1LInf}.
\end{proof}

This result allows us to obtain intermediate estimates between
 the results in \cite{FuTomi} (in which $2<p_j<\infty$ and $2<p<\infty$) and the results in
\cite{GrMiTo} (in which $1<p_j\le \infty$ and $1<p\le 2$).

\begin{cor}
Let $1<p_j\le \infty$ and $1<p<\infty$ satisfy  $1/p_1+\cdots +1/p_m=1/p$. Assume that \eqref{equ:SigmCond} holds for a   function $\sigma $ on $\mathbb R^{mn}$ where $s_i>n/2$ for all $i$. Then the   multilinear Fourier multiplier operator $T_\sigma$ maps $L^{p_1}\times
\cdots \times L^{p_m} $ to $L^p$
 \end{cor}

\begin{proof}
  Note that Sobolev condition \eqref{equ:SigmCond} for $\sigma$ implies Hormander condition \eqref{equ:HormCond} for $K=\sigma\spcheck.$ The proof of this implication is standard in the linear case and in the $m$-linear case it follows by freezing all but one variable (in the bilinear case it is contained in \cite{MiTo}). We are now able to apply Theorem \ref{thm:EndPts} to $T_\sigma,$ and hence Corollary \ref{last} follows.
 Interpolating between \eqref{equ:LInf2BMO} and \eqref{equ:LWL1LInf} yields that $T_\sigma$ maps $L^p\times L^\infty\times\cdots\times L^\infty$    to $L^p$ for all $1<p<\infty.$ By symmetry, we deduce that $T_\sigma$ is bounded from $L^{q_1}\times \cdots\times L^{q_m}$    to $L^p$ for all $(q_1,\ldots,q_m)\in\Lambda_p$ and $1<p<\infty.$ Once again, by interpolation, we have that $T_\sigma$ maps from $L^{p_1}\times\cdots L^{p_m}$ to $L^p$ for all \linebreak $1<p_1,\ldots,p_m\le \infty$ and $1<p<\infty$ such that $\frac1{p_1}+\cdots+\frac1{p_m} = \frac1{p}$ with norm at most a multiple of $A .$
  \end{proof}

 \begin{cor}\label{last}
 Let $\sigma$ be a bounded function on $\mathbb R^{mn}\setminus\{0\}$ which satisfies
 \eqref{equ:SigmCond} with $s_j>n/2$ for all $j=1,\dots , m$. Then we have the estimate
 \begin{equation}\label{equ:LInf2BMO22}
\| T_\sigma(f_1, \dots , f_m) \|_{BMO} \lesssim  A \| f_1\|_{L^\infty} \cdots  \| f_m\|_{L^\infty}
\end{equation}
for all functions $f_j\in L^\infty_c$.
 \end{cor}

\begin{proof}
  As before     condition \eqref{equ:SigmCond} for $\sigma$ implies   \eqref{equ:HormCond} for $K=\sigma\spcheck.$  Applying Theorem \ref{thm:EndPts} to $T_\sigma $,      Corollary \ref{last} follows.
\end{proof}

\section{Proofs of some technical lemmas}
In this section, we will give the detail proofs of some lemmas that were used in previous sections.
\subsection{The proof of Lemma \ref{lem:RegSobN}}
For $k\in\mathbb Z,k\ge 2$ denote $$F_k = \left\{y\in\mathbb R^{mn} :\ 2^{k-1}-2\le {\left\vert{y}\right\vert}\le 2^{k+1}+2\right\}.$$ Fix $x=(x_1,\ldots,x_m)\in\mathbb R^{mn}.$ Then we have
  \begin{align}
  (\varphi_\epsilon*\sigma)(2^jx)\widehat\psi(x) =& \left\{\int \epsilon^{-mn}\varphi
  \big(  {\epsilon}^{-1}y\big)\sigma(2^jx-y)dy\right\}\widehat\psi(x)\notag\\
  =&\left\{\int \epsilon^{-mn}2^{jmn}\varphi\big({\epsilon}^{-1} {2^jy} \big)\sigma(2^j(x-y))dy\right\}\widehat\psi(x)\notag\\
  =& \sum_{k\in\mathbb Z}\left\{\int \varphi_{\epsilon2^{-j}}(y)
  \sigma(2^j(x-y))\widehat{\psi}(2^{-k}(x-y))dy\right\}\widehat\psi(x)\notag\\
  \label{ali:ksmall}
  =&\sum_{k\le -3}\Big\{\int \varphi_{\epsilon2^{-j}}(x-y)
  \sigma(2^jy)\widehat{\psi}(2^{-k}y)dy\Big\}\widehat\psi(x)
  \\
  \label{ali:kmedium}
  &+\sum_{{\left\vert{k}\right\vert}\le 2}\int \varphi_{\epsilon2^{-j}}(y)
  \sigma(2^j(x-y))\widehat{\psi}(2^{-k}(x-y))\widehat\psi(x)dy
  \\
  \label{ali:klarge}
  &+\sum_{k\ge 3}\Big\{\int \varphi_{\epsilon2^{-j}}(x-y)
  \sigma(2^jy)\widehat{\psi}(2^{-k}y)dy\Big\}\widehat\psi(x).
    \end{align}
The $W^{(s_1,\ldots,s_m)}$ norm of   term \eqref{ali:kmedium} can be estimated easily by
\begin{eqnarray*}
&&\hspace{-.4in}\sum_{{\left\vert{k}\right\vert}\le 2} \int {\left\vert{\varphi_{\epsilon2^{-j}}(y)}\right\vert}
  \left\Vert{\sigma(2^j(\cdot-y))\widehat{\psi}(2^{-k}(\cdot-y))\widehat\psi}
  \right\Vert_{W^{(s_1,\ldots,s_m)}}dy\\
 & \le &\sum_{{\left\vert{k}\right\vert}\le 2}\int {\left\vert{\varphi_{\epsilon2^{-j}}(y)}\right\vert}
  \left\Vert{\sigma(2^j(\cdot-y))\widehat{\psi}(2^{-k}(\cdot-y))}\right\Vert_{W^{(s_1,\ldots,s_m)}}
  \left\Vert{\widehat\psi}\right\Vert_{W^{(s_1,\ldots,s_m)}}dy\\
 & \lesssim &\sum_{{\left\vert{k}\right\vert}\le 2}\left\Vert{\sigma(2^{j+k}\cdot)\widehat{\psi}}\right\Vert_{W^{(s_1,\ldots,s_m)}}\int {\left\vert{\varphi_{\epsilon2^{-j}}(y)}\right\vert}dy
\lesssim \sup_{j\in\mathbb Z}\left\Vert{\sigma(2^j\cdot)\widehat\psi}\right\Vert_{W^{(s_1,\ldots,s_m)}},
\end{eqnarray*}
in which the second last inequality follows from the fact \cite[Proposition A.2]{FuTomi} that
$$
\left\Vert{fg}\right\Vert_{W^{(s_1,\ldots,s_m)}}\lesssim \left\Vert{f}\right\Vert_{W^{(s_1,\ldots,s_m)}}\left\Vert{g}\right\Vert_{W^{(s_1,\ldots,s_m)}},
$$
when $f,g\in W^{(s_1,\ldots,s_m)}$   for $s_1,\ldots,s_m>\frac{n}2.$

Now fix integer numbers $l_i\ge s_i$ and set $l=l_1+\cdots+l_m.$ Since ${\left\Vert{f}\right\Vert}_{W^{(s_1,\ldots,s_m)}}\le {\left\Vert{f}\right\Vert}_{W^l},$ the $W^{(s_1,\ldots,s_m)}$ norm of the   term in  \eqref{ali:ksmall} is bounded  by
\begin{eqnarray*}
&&\hspace{-.2in}\sum_{k\le -3}\left\Vert{\bigg\{\int \varphi_{\epsilon2^{-j}}(\cdot-y)
\sigma(2^jy)\widehat{\psi}(2^{-k}y)dy\bigg\}\widehat\psi}\,\right\Vert_{W^l}\\
&\lesssim &\sum_{k\le -3}\sum_{{\left\vert{\alpha}\right\vert}+{\left\vert{\beta}\right\vert}\le l}\left\Vert{\bigg\{\int (\epsilon2^{-j})^{-{\left\vert{\alpha}\right\vert}}(\partial^{\alpha}\varphi)_{\epsilon2^{-j}}(\cdot-y)
\sigma(2^jy)\widehat{\psi}(2^{-k}y)dy\bigg\} \partial^{\beta}\widehat\psi }\, \right\Vert_{L^2}\\
&=& \sum_{k\le -3}\sum_{{\left\vert{\alpha}\right\vert}+{\left\vert{\beta}\right\vert}\le l}\left\Vert{\bigg\{\int_{\frac14\le {\left\vert{y}\right\vert}\le \frac94}  \frac{(\partial^{\alpha}\varphi)_{\epsilon2^{-j}}(y)}{(\epsilon2^{-j})^{ {\left\vert{\alpha}\right\vert}}  }
\sigma(2^j(\cdot-y))\widehat{\psi}(2^{-k}(\cdot-y))dy\bigg\} \partial^{\beta}\widehat\psi }\right\Vert_{L^2}\\
&\lesssim& \sum_{k\le -3}\sum_{{\left\vert{\alpha}\right\vert}\le l}\int \Big(\frac{{\left\vert{y}\right\vert}}{\epsilon2^{-j}}\Big)^{{\left\vert{\alpha}\right\vert}}
{\left\vert{(\partial^{\alpha}\varphi)_{\epsilon2^{-j}}(y)}\right\vert}
\left\Vert{\sigma(2^j(\cdot-y))\widehat{\psi}(2^{-k}(\cdot-y))}\right\Vert_{L^2}dy\\
&\lesssim& \sum_{k\le -3}2^{\frac{kn}2}
\left\Vert{\sigma(2^{j+k}\cdot)\widehat{\psi}}\right\Vert_{L^2}\sum_{{\left\vert{\alpha}\right\vert}\le l}\int \Big(\frac{{\left\vert{y}\right\vert}}{\epsilon2^{-j}}\Big)^{{\left\vert{\alpha}\right\vert}} \left\vert{(\partial^{\alpha}\varphi)_{\epsilon2^{-j}}(y)}\right\vert dy\\
&\lesssim& \sum_{k\le -3}2^{\frac{kn}2}
\left\Vert{\sigma(2^{j+k}\cdot)\widehat{\psi}}\right\Vert_{W^s}\sum_{{\left\vert{\alpha}\right\vert}\le l}\int {\left\vert{y}\right\vert}^{{\left\vert{\alpha}\right\vert}}
{\left\vert{(\partial^{\alpha}\varphi)(y)}\right\vert}dy \\
&\lesssim& \sup_{j\in\mathbb Z}{\left\Vert\sigma(2^j\cdot)\widehat\psi\right\Vert}_{W^{(s_1,\ldots,s_m)}}.
\end{eqnarray*}

Finally, we deal with term \eqref{ali:klarge}. We have
\begin{eqnarray*}
&&\hspace{-.3in}\left\Vert{\sum_{k\ge 3}\Big\{\int \varphi_{\epsilon2^{-j}}(\cdot-y)
\sigma(2^jy)\widehat{\psi}(2^{-k}y)dy\Big\}\widehat\psi}\,\right\Vert_{W^{(s_1,\ldots,s_m)}}\\
&\le&
\sum_{k\ge 3}\left\Vert{\Big\{\int \varphi_{\epsilon2^{-j}}(\cdot-y)
\sigma(2^jy)\widehat{\psi}(2^{-k}y)dy\Big\}\widehat\psi}\, \right\Vert_{W^l}\\
&\lesssim& \sum_{k\ge 3}\sum_{{\left\vert{\alpha}\right\vert}+{\left\vert{\beta}\right\vert}\le l}\left\Vert{\Big\{\int (\epsilon2^{-j})^{-{\left\vert{\alpha}\right\vert}}(\partial^{\alpha}\varphi)_{\epsilon2^{-j}}(\cdot-y)
\sigma(2^jy)\widehat{\psi}(2^{-k}y)dy\Big\} \partial^{\beta}\widehat\psi }\right\Vert_{L^2}\\
&=& \sum_{k\ge 3}\sum_{{\left\vert{\alpha}\right\vert}+{\left\vert{\beta}\right\vert}\le l}\left\Vert{\Big\{\int_{F_k} (\epsilon2^{-j})^{-{\left\vert{\alpha}\right\vert}}(\partial^{\alpha}\varphi)_{\epsilon2^{-j}}(y)
\sigma(2^j(\cdot-y))\widehat{\psi}(2^{-k}(\cdot-y))dy\Big\} \partial^{\beta}\widehat\psi }\, \right\Vert_{L^2}\\
&\lesssim& \sum_{k\ge 3}\sum_{{\left\vert{\alpha}\right\vert}+{\left\vert{\beta}\right\vert}\le l}\int_{F_k}\!\!\! 
\Big(\frac{{\left\vert{y}\right\vert}}{2^k\epsilon2^{-j}}  \!  \Big)^{{\left\vert{\alpha}\right\vert}}
{\left\vert{(\partial^{\alpha}\varphi)_{\epsilon2^{-j}}(y)}\right\vert}
\left\Vert{\sigma(2^j(\cdot-y))\widehat{\psi}(2^{-k}(\cdot-y)) \partial^{\beta}\widehat\psi }\,
\right\Vert_{L^2}dy\\
&=& \sum_{k\ge 3}\sum_{{\left\vert{\alpha}\right\vert}+{\left\vert{\beta}\right\vert}\le l}\int_{F_k} 
 \Big(\frac{{\left\vert{y}\right\vert}}{2^k\epsilon2^{-j}}\Big)^{{\left\vert{\alpha}\right\vert}}
{\left\vert{(\partial^{\alpha}\varphi)_{\epsilon2^{-j}}(y)}\right\vert}
\left\Vert{\sigma(2^j\cdot)\widehat{\psi}(2^{-k}\cdot)(\partial^{\beta}\widehat\psi)(\cdot+y)}
\right\Vert_{L^2}dy\\
&\le & \sum_{k\ge 3}\sum_{{\left\vert{\alpha}\right\vert}+{\left\vert{\beta}\right\vert}\le l}\int_{F_k} 2^{\frac{kn}2} \Big(\frac{{\left\vert{y}\right\vert}}{\epsilon2^{-j}}\Big)^{{\left\vert{\alpha}\right\vert}}
{\left\vert{(\partial^{\alpha}\varphi)_{\epsilon2^{-j}}(y)}\right\vert}
\left\Vert{\sigma(2^{j+k}\cdot)\widehat{\psi}(\partial^{\beta}\widehat\psi)(2^k\cdot+y)}
\right\Vert_{L^2}dy\\
&\le& \sum_{k\ge 3}\sum_{{\left\vert{\alpha}\right\vert}\le l}\int_{F_k} 2^{\frac{kn}2 } \Big(\frac{{\left\vert{y}\right\vert}}{\epsilon2^{-j}}\Big)^{{\left\vert{\alpha}\right\vert}}
{\left\vert{(\partial^{\alpha}\varphi)_{\epsilon2^{-j}}(y)}\right\vert}
\left\Vert{\sigma(2^{j+k}\cdot)\widehat{\psi}}\right\Vert_{L^2(B(-2^{-k}y,2^{1-k}))}dy\\
&\lesssim& \sum_{{\left\vert{\alpha}\right\vert}\le l}\sum_{k\ge 3}
\left\Vert{\sigma(2^{j+k}\cdot)\widehat{\psi}}\, \right\Vert_{L^\infty} 
\int_{F_k} \Big(\frac{{\left\vert{y}\right\vert}}{\epsilon2^{-j}}\Big)^{{\left\vert{\alpha}\right\vert}} {\left\vert{(\partial^{\alpha}\varphi)_{\epsilon2^{-j}}(y)}\right\vert}dy\\
&\le& \sum_{{\left\vert{\alpha}\right\vert}\le l}\sum_{k\ge 3}
\left\Vert{\sigma(2^{j+k}\cdot)\widehat{\psi}}\, \right\Vert_{W^{(s_1,\ldots,s_m)}}
\int_{F_k} \Big(\frac{{\left\vert{y}\right\vert}}{\epsilon2^{-j}}\Big)^{{\left\vert{\alpha}\right\vert}} {\left\vert{(\partial^{\alpha}\varphi)_{\epsilon2^{-j}}(y)}\right\vert}dy\\
&\lesssim& \sup_{k\in\mathbb Z}\left\Vert{\sigma(2^{j+k}\cdot)\widehat{\psi}}\, \right\Vert_{W^{(s_1,\ldots,s_m)}}
\sum_{{\left\vert{\alpha}\right\vert}\le l}\sum_{k\ge 3}
\int_{F_k} \Big(\frac{{\left\vert{y}\right\vert}}{\epsilon2^{-j}}\Big)^{{\left\vert{\alpha}\right\vert}} {\left\vert{(\partial^{\alpha}\varphi)_{\epsilon2^{-j}}(y)}\right\vert}dy\\
&\lesssim& \sup_{j\in\mathbb Z}{\left\Vert\sigma(2^j\cdot)\widehat\psi\, \right\Vert}_{W^{(s_1,\ldots,s_m)}}.
\end{eqnarray*}
The proof of the lemma is now complete.
\subsection{The proof of Lemma \ref{lem:KeyLem}}
Before verifying Lemma \ref{lem:KeyLem}, we mention   approaches that were used by other authors. First, with assumption on the kernel
  $$
  {\left\vert{\partial_{y_0,\ldots,y_m}^{\alpha}K(y_0,\ldots,y_m)}\right\vert}\le A\left(\sum_{k,l=0}^m{\left\vert{y_k-y_l}\right\vert}\right)^{-mn-{\left\vert{\alpha}\right\vert}}
  $$
  for all ${\left\vert{\alpha}\right\vert}\le N,$ Grafakos and Kalton \cite{GrKal} showed that estimate \eqref{equ:BiFuncs} holds for the corresponding multilinear singular integral operator with
  $$
  b_i(x)=\dfrac{{\left\vert{Q_i}\right\vert}^{1-\frac{1}{p_i}+\frac{N+1}{mn}}}
  {\Big({\left\vert{x-c_i}\right\vert}+\ell(Q_i)\Big)^{n+\frac{N+1}{m}}}.
  $$
Miyachi and Tomita \cite{MiTo} constructed functions $b_i$ satisfying Lemma \ref{lem:KeyLem} in the bilinear case. We adapt these techniques to prove the key lemma in the  multilinear setting.

  Now we start the proof of Lemma \ref{lem:KeyLem}.
   We may assume that $J_0=\left\{1,\ldots,r\right\}$ for some $1\le r\le m.$ Fix
   $$x\in \Big(\bigcap_{i=r+1}^mQ_i^*\Big)\setminus\bigcup_{i=1}^rQ_i^*$$
   (when $r=m,$ just fix $x\in \mathbb R^n\setminus\bigcup_{i=1}^mQ_i^*$).
   Now we rewrite $T_{\sigma}(a_1,\ldots,a_m)(x)$ as
    $$
    T_{\sigma}(a_1,\ldots,a_m)(x) = \sum_{j\in\mathbb Z}g_j(x),
    $$
    where
    $$
    g_j(x) = \int_{\mathbb R^{mn}}2^{jmn}K_j(2^j(x-y_1), \ldots, 2^j(x-y_m))a_1(y_1)\cdots a_m(y_m)dy_1\cdots dy_m
    $$
    with $K_j = \big(\sigma(2^j\cdot)\widehat\psi\, \big)\spcheck.$
    Let $c_i$ be the center of the cube $Q_i$ $(1\le i\le m).$
    For $1\le i\le r,$ since $x\notin Q_i^*$ and $y_i\in Q_i,$ ${\left\vert{x-c_i}\right\vert}\approx {\left\vert{x-y_i}\right\vert}.$
    Fix $1\le k\le r.$
    Using Lemma \ref{lem:LInfL2}  with $s_i>\frac{n}{2} $ and applying the Cauchy-Schwarz inequality we obtain
    \begin{align}
      \prod_{i=1}^r&\left<2^j(x-c_i)\right>^{s_i}{\left\vert{g_j(x)}\right\vert}\notag\\
      \lesssim &\ 2^{jmn}\!\! \! \int\limits_{Q_1\times\cdots\times{Q_{m}}}
      \prod_{i=1}^r\left<2^j(x-y_i)\right>^{s_i}{\left\vert{K_j(2^j(x-y_1),\ldots,2^j(x-y_m))}\right\vert} \prod_{i=1}^m\norm{a_i}_{L^{\infty}}d\vec y\notag\\
      \le&\ 2^{jmn}\!\! \! \int\limits_{Q_1\times\cdots\times{Q_{m}}}
      \prod_{i=1}^r\left<2^j(x-y_i)\right>^{s_i}{\left\vert{K_j(2^j(x-y_1),\ldots,2^j(x-y_m))}\right\vert}
      \prod_{i=1}^m{\left\vert{Q_i}\right\vert}^{-\frac1{p_i}}d\vec y\notag\\
      \le&\ 2^{jrn}\prod_{i=1}^m{\left\vert{Q_i}\right\vert}^{-\frac1{p_i}}
      \int\limits_{Q_1\times\cdots\times{Q_r}\times \mathbb R^{(m-r)n}}
      \prod_{i=1}^r\left<2^j(x-y_i)\right>^{s_i}\times\notag\\
      &\times{\left\vert{K_j(2^j(x-y_1),\ldots,2^j(x-y_r),y_{r+1},\ldots,y_m)}\right\vert}dy_1\cdots dy_rdy_{r+1}\cdots dy_m\notag\\
      \le &\  2^{jrn}\prod_{i=1}^r{\left\vert{Q_i}\right\vert}^{1-\frac1{p_i}}
      \prod_{i=r+1}^m{\left\vert{Q_i}\right\vert}^{-\frac1{p_i}}\int_{\mathbb R^{(m-r)n}}
      \int\limits_{Q_k} {\left\vert{Q_k}\right\vert}^{-1}\left<2^j(x-y_k)\right>^{s_k}\times\notag\\
      &\times\norm{\prod_{\substack{i=1\\i\ne k}}^{r}\left<y_i\right>^{s_i}K_j(y_1,\ldots,y_{k-1},2^j(x-y_k),y_{k+1}, \ldots,y_m)}_{L^{\infty}(dy_1 \cdots \widehat{dy_k} \cdots dy_{r})}
   \hspace{-.8in}   dy_k dy_{r+1}\cdots dy_m\notag\\
      \lesssim &\ 2^{jrn}\prod_{i=1}^r{\left\vert{Q_i}\right\vert}^{1-\frac1{p_i}}
      \prod_{i=r+1}^m{\left\vert{Q_i}\right\vert}^{-\frac1{p_i}}\int_{\mathbb R^{(m-r)n}}
      \int\limits_{Q_k}
      {\left\vert{Q_k}\right\vert}^{-1}\left<2^j(x-y_k)\right>^{s_k}\times\notag\\
      &\times\norm{\prod_{\substack{i=1\\i\ne k}}^{r}\left<y_i\right>^{s_i}K_j(y_1,\ldots,y_{k-1},2^j(x-y_k),y_{k+1}, \ldots,y_m)}_{L^{2}(dy_1 \cdots \widehat{dy_k} \cdots dy_{r})}
      \hspace{-.8in} dy_kdy_{r+1}\cdots dy_m\notag\\
      \lesssim &\ 2^{jrn}\prod_{i=1}^r{\left\vert{Q_i}\right\vert}^{1-\frac1{p_i}}
      \prod_{i=r+1}^m{\left\vert{Q_i}\right\vert}^{-\frac1{p_i}}
      \int\limits_{Q_k}
      {\left\vert{Q_k}\right\vert}^{-1}\left<2^j(x-y_k)\right>^{s_k}\times\notag\\
      &\times\norm{\prod_{\substack{i=1\\i\ne k}}^{m}\left<y_i\right>^{s_i}K_j(y_1,\ldots,y_{k-1},2^j(x-y_k),y_{k+1}, \ldots,y_m)}_{L^{2}(dy_1 \cdots \widehat{dy_k} \cdots dy_m)}
      dy_k\notag\\
      \label{ali:H10Func}
      =&\ 2^{jrn}\prod_{i=1}^r{\left\vert{Q_i}\right\vert}^{1-\frac1{p_i}}h_j^{(k,0)}(x)\prod_{i=r+1}^mb_i(x)
    \end{align}
    for all $x\in (\cap_{i=r+1}^mQ_i^*)\setminus (\cup_{i=1}^rQ_i^*),$
    where
    \begin{align*}
    h_j^{(k,0)}(x) = &\, \dfrac1{{\left\vert{Q_k}\right\vert}}\int\limits_{Q_k}\left<2^j(x-y_k)\right>^{s_k}\\
      & \hspace{-.5in}\times
      \norm{\prod_{\substack{i=1\\i\ne k}}^{m}\left<y_i\right>^{s_i}K_j(y_1,\ldots,y_{k-1},2^j(x-y_k),y_{k+1}, \ldots,y_m)}_{L^{2}(
      dy_1 \cdots \widehat{dy_k} \cdots dy_m)}dy_k
    \end{align*}
    and $b_i(x) = {\left\vert{Q_i}\right\vert}^{-\frac1{p_i}}\chi_{Q_i^*}(x)$ for $r+1\le i\le m.$
    A direct computation gives
    $$
    \norm{h_j^{(k,0)}}_{L^2}\le 2^{-\frac{jn}{2}}{\left\Vert\sigma(2^j\cdot)\widehat\psi\right\Vert}_{W^{(s_1,\ldots,s_m)}} = A 2^{-\frac{jn}{2}}.
    $$
    Using the vanishing moment condition of $a_k$ and Taylor's formula, we write
    $$
    \begin{aligned}
    g_j(x) \, = &\, 2^{jmn}\sum_{{\left\vert{\alpha}\right\vert}=N_k}C_{\alpha}\int_{\mathbb R^{mn}}
    \Big\{\int_0^1(1-t)^{N_k-1}\\
    &\times\partial^{\alpha}_{y_k}K_j\Big(2^j(x-y_1),\ldots,2^j(x-c_k-t(y_k-c_k)),\ldots, 2^j(x-y_m)\Big)\\
    &\times(2^j(y_k-c_k))^{\alpha}a_1(y_1)\cdots a_m(y_m)\ dt\Big\}\ dy_1\cdots dy_m.
    \end{aligned}
    $$
    Repeat the preceding argument to obtain
    \begin{equation}\label{equ:H1NFunc}
    \prod_{i=1}^r\left<2^j(x-c_i)\right>^{s_i}{\left\vert{g_j(x)}\right\vert}\lesssim 2^{jrn}\prod_{i=1}^r{\left\vert{Q_i}\right\vert}^{1-\frac1{p_i}} h_j^{(k,1)}(x)\prod_{i=r+1}^mb_i(x)
    \end{equation}
    for all $x\in (\cap_{i=r+1}^mQ_i^*)\setminus (\cup_{i=1}^rQ_i^*),$
    where $b_i(x) = {\left\vert{Q_i}\right\vert}^{-\frac1{p_i}}\chi_{Q_i^*}(x)$ for $r+1\le i\le m$ and
    $$
    \begin{aligned}
    h_j^{(k,1)}(x) = &\sum_{{\left\vert{\alpha}\right\vert}=N_k}\int_{Q_k}\Big\{\int_0^1\left<2^jx_{c_k,y^k}^t\right>^{s_k}\\
    &\times\norm{
    \prod_{\substack{i=1\\i\ne k}}^{m}\left<y_i\right>^{s_i} {\partial^{\alpha}_{y_k}K_j(y_1,\ldots,y_{k-1},2^jx_{c_k,y^k}^t,y_{k+1},\ldots,y_m)}}_{L^2(dy_1  \cdots \widehat{dy_k} \cdots dy_m)}\\
    &\times(2^j\ell(Q_k))^{N_k}{\left\vert{Q_k}\right\vert}^{-1}\ dt\Big\}\ dy_k,
    \end{aligned}
    $$
    with $x_{c_k,y_k}^t=x-c_k-t(y_k-c_k).$
    Applying Minkowski's inequality together with Lemma \ref{lem:ProdSobN} implies that
    $$
    \norm{h_j^{(k,1)}}_{L^2}\le A 2^{-\frac{jn}{2}}(2^j\ell(Q_k))^{N_k}.
    $$
    Combine inequalities \eqref{ali:H10Func} and \eqref{equ:H1NFunc}, we get
    \begin{equation}\begin{split} \label{equ:GjQ1}
     & \prod_{i=1}^r\left<2^j(x-c_i)\right>^{s_i}{\left\vert{g_j(x)}\right\vert}
      \\
    & \hspace{1in} \le 2^{jrn}\prod_{i=1}^r{\left\vert{Q_i}\right\vert}^{1-\frac1{p_i}} \min\left\{h_j^{(k,0)}(x),h_j^{(k,1)}(x)\right\}\prod_{i=r+1}^mb_i(x)
    \end{split}\end{equation}
    for all $1\le k\le r.$ The inequalities in \eqref{equ:GjQ1} imply that
    \begin{equation}\begin{split}\label{equ:GjQ2}
   & {\left\vert{g_j(x)}\right\vert}  \\
  &   \le 2^{jrn}\prod_{i=1}^r{\left\vert{Q_i}\right\vert}^{1-\frac1{p_i}} \prod_{i=1}^r\left<2^j(x-c_i)\right>^{-s_i}
  \!\!\min_{1\le k\le r}\left\{h_j^{(k,0)}(x),h_j^{(k,1)}(x)\right\}\!\! \prod_{i=r+1}^m \!\! b_i(x)
    \end{split}\end{equation}
    for all $x\in (\cap_{i=r+1}^mQ_i^*)\setminus (\cup_{i=1}^rQ_i^*).$

    Now we need to construct functions $u^k_j$ $(1\le k \le r)$ such that $$g_j(x)\lesssim A\prod_{k=1}^ru^k_j(x)\prod_{i=r+1}^mb_i(x)$$ for all  $x\in (\cap_{i=r+1}^mQ_i^*)\setminus (\cup_{i=1}^rQ_i^*)$ and that
    $\norm{\sum_{j}u^k_j}_{L^{p_k}}\lesssim 1$ for all $1\le k\le r.$ Then the lemma follows by taking $b_k=\sum_{j}u^k_j$ ($1\le k\le r$) and $b_i = {\left\vert{Q_i}\right\vert}^{-\frac1{p_i}}\chi_{Q_i^*}$ $(r+1\le i\le m).$

    Indeed, we can choose $0<\lambda_k <\min\Big\{\frac12,\frac{s_k}{n}-\frac{1}{p_k}+\frac12\Big\}$ such that $$\sum_{k=1}^r\lambda_k= \dfrac{r-1}2.$$
    This is suitable since conditions \eqref{equ:SIndexCond} implies that
    $$
    \sum_{k=1}^r\min\Big\{\frac12,\frac{s_k}{n}-\frac{1}{p_k}+\frac12\Big\}>\frac{r-1}2.
    $$
    Set $\alpha_k=\frac1{p_k}-\frac12$ and $\beta_k = 2(\frac1{p_k}-\alpha_k).$
    Then we have
    $$\sum_{k=1}^r\alpha_k=\sum_{k=1}^r\dfrac1{p_k}-\dfrac12,$$
    $\beta_k>0$ and $\beta_1+\cdots+\beta_r=1.$
    Now define
    $$
    u^k_j = A^{-\beta_k}2^{jn}{\left\vert{Q_k}\right\vert}^{1-\frac1{p_k}}
    \left<2^j(\cdot-c_k)\right>^{-s_k}\chi_{(Q_k^*)^c}
 \min\left\{h_j^{(k,0)},h_j^{(k,1)}\right\}^{\beta_k},\quad 1\le k\le r .
    $$
    Then, from \eqref{equ:GjQ2}, it is easy to see that
    $$
    g_j(x)\lesssim A\prod_{k=1}^ru^k_j(x)\prod_{i=r+1}^mb_i(x)
    $$
     for all  $x\in (\cap_{i=r+1}^mQ_i^*)\setminus (\cup_{i=1}^rQ_i^*).$ It remains to check that
    $\sum_{j}\int_{\mathbb R^n}{\left\vert{u^k_j(x)}\right\vert}^{p_k}dx\lesssim 1.$

    Since $\dfrac1{p_k} = \alpha_k+\dfrac{\beta_k}2,$ setting $\frac{1}{p_k'}=1-\frac1{p_k}$,   H\"older's inequality gives
    \begin{align*}
    \norm{u^k_j}_{L^{p_k}} \le A^{-\beta_k}2^{jn}{\left\vert{Q_k}\right\vert}^{ \frac1{p_k'}}
    \norm{\left<2^j(\cdot-c_k)\right>^{-s_k}\chi_{(Q_k^*)^c}}_{L^{\frac{1}{\alpha_k}}}
    \norm{ \min\left\{h_j^{(k,0)},h_j^{(k,1)}\right\} ^{\beta_k}}_{L^{\frac{2}{\beta_k}}}
    \end{align*}
    for all $1\le k\le r.$
    Notice that $n<\frac{s_k}{\alpha_k},$ we have
    $$
    \norm{\left<2^j(\cdot-c_k)\right>^{-s_k}\chi_{(Q_k^*)^c}}_{L^{1/\alpha_k}}\lesssim 2^{-jn\alpha_k} \min\left\{1,(2^j\ell(Q_k))^{\alpha_kn-s_k}\right\}
    $$
    and
    $$
    \begin{aligned}
    \norm{\left(\min\left\{h_j^{(k,0)},h_j^{(k,1)}\right\}\right)^{\beta_k}}_{L^{2/\beta_k}}
    \le &\min\left\{\norm{h_j^{(k,0)}}_{L^2}^{\beta_k},\norm{h_j^{(k,1)}}_{L^2}^{\beta_k}\right\}\\
    \lesssim & \Big(A2^{-jn/2}\min\left\{1,(2^j\ell(Q_k))^{N_k}\right\}\Big)^{\beta_k}.
    \end{aligned}
    $$
    Therefore
    $$
    \begin{aligned}
    \norm{u^k_j}_{L^{p_k}}\le& 2^{jn}{\left\vert{Q_k}\right\vert}^{1-\frac1{p_k}}2^{-\frac{jn}{p_k}}
    \min\left\{1,(2^j\ell(Q_k))^{\alpha_kn-s_k}\right\}
    \min\left\{1,(2^j\ell(Q_k))^{N_k\beta_k}\right\}\\
    \lesssim & (2^j\ell(Q_k))^{n-\frac{n}{p_k}}
    \min\left\{1,(2^j\ell(Q_k))^{\alpha_kn-s_k}\right\}
    \min\left\{1,(2^j\ell(Q_k))^{N_k\beta_k}\right\}.
    \end{aligned}
    $$
    This inequality is enough to establish what we needed
    $
    \sum_{j\in\mathbb Z}\int_{\mathbb R^n}{\left\vert{u^k_j(x)}\right\vert}^{p_k}dx\lesssim 1.
    $
The proof of Lemma \ref{lem:KeyLem} is complete.

\bibliographystyle{amsplain}

\end{document}